%% file: ms.tex
\documentclass[11pt]{amsart}
\usepackage{amssymb,mathtools}
\usepackage{graphicx,microtype,booktabs,array,float}
\usepackage[section]{placeins}
\usepackage[margin=1in]{geometry}
\usepackage[hidelinks,hypertexnames=false]{hyperref}
\newtheorem{theorem}{Theorem}[section]
\newtheorem{lemma}[theorem]{Lemma}
\newtheorem{proposition}[theorem]{Proposition}

\theoremstyle{remark}
\newtheorem{remark}[theorem]{Remark}

\newcommand{\N}{\mathbb N}
\newcommand{\e}{\mathrm e}
\newcommand{\ii}{\mathrm i}
\newcommand{\norm}[1]{\lVert#1\rVert}
\newcommand{\Sig}{\Sigma_\Theta}
\newcommand{\doiurl}[1]{\href{https://doi.org/#1}{\nolinkurl{doi:#1}}}
\input{numerical_macros.tex}
\input{dimension_macros.tex}
\input{exact_asymptotic_macros.tex}
\input{space_benchmark_macros.tex}
\input{supplement_macros.tex}

\hypersetup{pdftitle={Shifted tapered-pole approximation on finite sectors: an exact asymptotic and dimension control},pdfauthor={Fei Xue and Tianqi Zhang}}

\title[Shifted tapered-pole approximation]{Shifted tapered-pole approximation
on finite sectors: an exact asymptotic and dimension control}

\author{Fei Xue}
\address{School of Mathematical and Statistical Sciences, Clemson University,
Clemson, South Carolina 29634, USA}
\email{fxue@clemson.edu}
\thanks{Fei Xue was supported in part by the National Science Foundation under
grant DMS-2608905.}

\author{Tianqi Zhang}
\address{School of Data Sciences, Zhejiang University of Finance and Economics,
18 Xueyuan Street, Hangzhou 310018, China}
\email{tianqiz@zufe.edu.cn}

\begin{document}
\raggedbottom

\begin{abstract}
We study explicit lightning-plus-polynomial (LP) approximation of $g(z)z^\alpha(\log z)^m$ on a closed unit-radius sector, where $0<\alpha<1$, $m$ is fixed, and $g$ is analytic nearby.  Since the basic sector root-exponential rate is already known, we address three finer questions.  First, a one-sided Abel--Poisson formula and a pole-free notched contour give a branchwise error bound that is uniform for grid shifts in compact subsets of $(0,1)$ and valid for every fixed logarithmic power.  Second, for $z^\alpha$ in the dense and balanced regimes, an exact vertex-scale rescaling identifies a positive finite leading constant for a specified shifted sequence.  Third, a rate-matched band of additional representation poles followed by polynomial compression gives the explicit representation dimension $N+2\sqrt{2\alpha\lambda_\sigma N}+O(1)$ while leaving the first $N$ tapered poles unchanged.  A slightly overresolved hybrid sequence also preserves the exact leading constant.  Matched-dimension experiments, including reentrant sectors, illustrate the allocation.  These exact-asymptotic and dimension statements concern the stated constructions, not sector minimax approximation or total computational complexity.
\end{abstract}


\maketitle

\section{The approximation problem and the results}\label{sec:intro}
We study a concrete lightning-plus-polynomial construction for targets of the form
$g(z)z^\alpha(\log z)^m$ on a closed sector of unit radius, with $0<\alpha<1$, a fixed nonnegative integer $m$, and analytic $g$.  By an LP approximant we mean a finite sum of simple rational terms at prescribed poles, together with a polynomial.  The poles resolve the branch singularity and the polynomial represents the smoother remainder.  Since the basic sector root-exponential rate is already known, our questions concern the structure beneath that rate: how an interior grid shift affects a rigorous guarantee, what leading error an explicit shifted sequence actually has, and how the rational and polynomial degrees should share the available representation dimension.

There is a long approximation-theoretic background to these questions. Newman's construction for $|x|$~\cite{Newman64} and its sharpening by Vyacheslavov~\cite{Vyacheslavov74} are classical examples of root-exponential rational approximation, while Stahl's theory identifies sharp errors and pole distributions for best approximation \cite{Stahl94,Stahl03}.  Explicit linear rational approximation with preassigned poles goes back at least to Stenger~\cite{Stenger86}, and the same singular behavior appears in the analysis of corner problems; see, for example, Wasow~\cite{Wasow57}.

The more recent ``lightning'' viewpoint turns this theory into a numerical construction: choose poles near the singularity, solve only for their coefficients, and add a polynomial for the smooth remainder.  This idea has been effective for Laplace problems with corners~\cite{GT19Laplace}, conformal mapping~\cite{GT19Conformal}, potential flow~\cite{Baddoo20}, Stokes flow~\cite{BT22}, and reciprocal-log PDE solvers~\cite{NT21}. Trefethen, Nakatsukasa, and Weideman~\cite{TNW21} connected tapered pole clustering with rational approximation and quadrature, while Herremans, Huybrechs, and Trefethen~\cite{HHT23} emphasized the usefulness of the polynomial part.  Interval, sector, and corner-domain convergence laws have been developed in \cite{XYW26,XY24,XXYZ26}; multivariate analogues are studied in \cite{BHH24}.

A particularly close comparison is the concurrent work of Xiang et al.~\cite{XXYZ26}, developed independently of the present paper within the same LP framework of \cite{HHT23}. In their notation, the full-opening fraction is $\beta=2\Theta/\pi$, and the balanced parameter is $\sigma_*=2\sqrt{\pi(\pi-\Theta)/\alpha}$. On the common range $0<\alpha<1$ and $m=0,1$, \cite{XXYZ26} establishes the sector rate and the corresponding branchwise algebraic orders. We therefore do not claim those overlapping rate statements as new here.  The two works pursue different refinements of the common sector-approximation problem. The present paper develops compact-shift uniformity and all fixed logarithmic powers at the upper-bound level, a specified-sequence exact constant in the dense/balanced regime, and an explicit leading coefficient for the rational--polynomial dimension split. Conversely, \cite{XXYZ26} treats a broader exponent range and multi-corner domains, which are outside our scope. The constructions and arguments developed here are independent of those in \cite{XXYZ26}.

The reader can organize the new results around three statements.  \emph{First}, Theorem~\ref{thm:main} gives one shift-uniform proof for every fixed logarithmic power while preserving the known branchwise exponent.  \emph{Second}, Theorem~\ref{thm:exact} goes beyond an upper bound for $z^\alpha$ in the dense and balanced regimes: it identifies an explicit shifted sequence whose normalized error converges to a positive, finite constant determined by a vertex-scale profile.  \emph{Third}, Proposition~\ref{prop:dimension} addresses representation cost.  For $N$ core poles, retaining a short band of additional representation/compression poles (called transition poles below) and matching polynomial compression to the active error branch gives total dimension $N+2\sqrt{2\alpha\lambda_\sigma N}+O(1)$, where $\lambda_\sigma=\min\{1,(\sigma_*/\sigma)^2\}$.  Retaining remote poles or adding a polynomial is not itself new~\cite{HHT23}; the contribution here is the explicit threshold and dimension coefficient while the original first $N$ poles remain unchanged.  A slightly overresolved realization preserves the exact leading constant as well.  These are statements about specified constructions, not the best error over all LP spaces.

The three statements are linked by one analytic mechanism: a one-sided Abel formula followed by a local notch in the square-root variable. The shifted rule leaves the integrable endpoint singularity untouched; the notch avoids enclosing the moving poles without sacrificing the nearest-pole exponent. This construction provides an alternative to the endpoint extension and residue bookkeeping used in~\cite{XYW26,XXYZ26}, while retaining the same error orders on the common range. The classical Poisson and contour tools themselves are standard~\cite{DLMF,Stenger11,TW14,JT14}; their role here is to preserve the structure needed for the profile and dimension analysis.

Let us fix
\begin{equation}\label{eq:domain}
0<\alpha<1,\qquad 0\le\Theta<\pi,\qquad
 \Sig=\{r\e^{\ii\theta}:0\le r\le1,\ |\theta|\le\Theta\}.
\end{equation}
Here and below $\mathbb N_0=\{0,1,2,\ldots\}$.  We use the principal logarithm and set $z^\alpha(\log z)^m=0$ at $z=0$ by continuity.  Write $\norm{f}_{\Sig}=\sup_{z\in\Sig}|f(z)|$.  The shift lies in a fixed compact interval $D\Subset(0,1)$.  For $\sigma>0$ and $\delta\in D$, prescribe
\begin{equation}\label{eq:poles}
 \xi_{k,N}^{(\delta)}
=-\exp\!\left[-\sigma\left(\sqrt N-\sqrt{k-\delta}\right)\right],
 \qquad 1\le k\le N.
\end{equation}
Write
\begin{equation}\label{eq:rates}
 \sigma_* =2\sqrt{\frac{\pi(\pi-\Theta)}\alpha},\qquad
c_\Theta(\sigma)=\min\left\{\alpha\sigma,
             \frac{4\pi(\pi-\Theta)}\sigma\right\}.
\end{equation}
We call $\sigma<\sigma_*$, $\sigma=\sigma_*$, and $\sigma>\sigma_*$ the dense, balanced, and sparse branches, respectively.

\begin{theorem}[Shift-uniform LP approximation]\label{thm:main}
Let $m\in\N_0$ be fixed and let $g$ be analytic on a neighborhood of $\Sig$. For each $N$ sufficiently large and every $\delta\in D$, there is an LP function
\begin{equation}\label{eq:LP}
r_{N,\delta}(z)=\sum_{k=1}^N
      \frac{d_{k,N,\delta}}{z-\xi_{k,N}^{(\delta)}}+p_{N,\delta}(z),
 \qquad \deg p_{N,\delta}\le C_p\sqrt N,
\end{equation}
such that
\begin{equation}\label{eq:main-bound}
 \norm{g(z)z^\alpha(\log z)^m-r_{N,\delta}}_{\Sig}
 \le C\begin{cases}
 (1+\sqrt N)^m\e^{-\alpha\sigma\sqrt N},&0<\sigma\le\sigma_*,\\
 \e^{-4\pi(\pi-\Theta)\sqrt N/\sigma},&\sigma>\sigma_*.
 \end{cases}
\end{equation}
Here $C_p>0$ and $C>0$ are independent of $N$ and $\delta$, but may depend on $\alpha,\Theta,\sigma,m,g,D$.  In the second branch no uniformity as $\sigma\downarrow\sigma_*$ is asserted.  Consequently the guaranteed leading exponent is $c_\Theta(\sigma)$, uniformly over $D$, and its largest value is
\begin{equation}\label{eq:balanced}
c_\Theta(\sigma_*)=2\sqrt{\pi\alpha(\pi-\Theta)}.
\end{equation}
The same leading exponent holds with $N$ replaced by the LP dimension $L_{N,\delta}=N+n_{N,\delta}+1=N+O(\sqrt N)$, where $n_{N,\delta}:=\deg p_{N,\delta}$ is the polynomial degree.  This count includes all prescribed basis functions, even when a coefficient vanishes.
\end{theorem}

For $m=0$ the prefactor in both branches is a constant; for $m=1$ it is $O(\sqrt N)$ only in the first branch, matching the algebraic orders of the endpoint construction in~\cite[Theorems 1.1--1.2]{XXYZ26} over the common range $0<\alpha<1$.  Theorem~\ref{thm:main} additionally gives uniformity over compact interior-shift families and covers every fixed $m$ through a one-sided proof that leaves the integrable endpoint singularity untouched.  The $O(\sqrt N)$ notation in Theorem~\ref{thm:main} also hides a practical question: how should rational and polynomial degrees be divided? Proposition~\ref{prop:dimension} answers this for the explicit singular prototypes by retaining a short transition-pole band and matching its compression accuracy to the active branch.  With $\lambda_\sigma=\min\{1,(\sigma_*/\sigma)^2\}$, the optimized split has total dimension $N+2\sqrt{2\alpha\lambda_\sigma N}+O(1)$, with the original first $N$ tapered poles untouched.  In particular, the balanced and dense branches have $\lambda_\sigma=1$, whereas the leading correction decreases by the factor $\sigma_*/\sigma$ in the sparse branch.

The upper bound can be sharpened further for the explicit algebraic prototype. The point is not to obtain a minimax lower bound, but to identify the actual leading error of this particular shifted construction.  Set $c_\alpha=\sin(\pi\alpha)/(\pi\alpha)$ and, for fixed $h>0$ and $0<\delta<1$, put $w_k=\sqrt{(k-\delta)h}$.  For $X\in\mathbb R$ and $|\theta|\le\Theta$, define the vertex-scale profile
\begin{equation}\label{eq:profile}
 \mathcal H_{h,\delta}(X,\theta)
=\e^{X+\ii\alpha\theta} -\frac{c_\alpha h}{2}\sum_{k=1}^\infty
 \frac{\e^{w_k+\ii\theta}}
{w_k\bigl(\e^{(w_k-X)/\alpha}+\e^{\ii\theta}\bigr)} .
\end{equation}
The fixed dependence on $\alpha$ is suppressed in $\mathcal H_{h,\delta}$ and in the constant below.

\begin{theorem}[Exact asymptotic for the explicit dense/balanced construction]
\label{thm:exact}Fix $0<\delta<1$ and $0<\sigma\le\sigma_*$, and set $h=\alpha^2\sigma^2$. The series in \eqref{eq:profile} is locally uniformly convergent on $\mathbb R\times[-\Theta,\Theta]$, and
\begin{equation}\label{eq:profile-constant}
0<C_{h,\delta,\Theta}:=
 \sup_{X\in\mathbb R,\ |\theta|\le\Theta}
|\mathcal H_{h,\delta}(X,\theta)|<\infty.
\end{equation}
There is an explicit sequence of LP functions $\widetilde r_{N,\delta}$ with exactly the poles \eqref{eq:poles} and polynomial degree $O(\sqrt N)$ such that
\begin{equation}\label{eq:exact-asymptotic}
 \norm{z^\alpha-\widetilde r_{N,\delta}}_{\Sig}
=\bigl(C_{h,\delta,\Theta}+o(1)\bigr)
   \e^{-\alpha\sigma\sqrt N}.
\end{equation}
Moreover, the same asymptotic is attained by an explicit hybrid sequence $\widetilde r^{\rm hyb}_{N,\delta}$ that keeps these $N$ core poles, adds only transition poles to the left of $-1$, and has total LP dimension
\begin{equation}\label{eq:exact-dimension}
L_N=N+2\sqrt{2\alpha N}+O(\log N).
\end{equation}
At $\sigma=\sigma_*$ both constructions attain the exact leading exponential rate $2\sqrt{\pi\alpha(\pi-\Theta)}$.
\end{theorem}

For the hybrid sequence in Theorem~\ref{thm:exact}, the exact constant depends on the convention used to measure degree.  Equation~\eqref{eq:exact-dimension} gives
\[
 \sqrt{L_N}-\sqrt N\longrightarrow\sqrt{2\alpha},
 \qquad
 \e^{\alpha\sigma\sqrt{L_N}}\norm{z^\alpha-\widetilde r^{\rm hyb}_{N,\delta}}_{\Sig}
 \longrightarrow C_{h,\delta,\Theta}\e^{\alpha\sigma\sqrt{2\alpha}}.
\]
Thus replacing the core-pole count $N$ by the total LP dimension preserves the root-exponential exponent but not the normalized leading constant.  We state \eqref{eq:exact-asymptotic} in terms of $N$ and report total dimensions separately whenever constants are compared.

The profile has a simple geometric meaning.  Set $T=\alpha\sigma\sqrt N$.  Writing $z=\exp((X-T)/\alpha+\ii\theta)$ keeps $X$ fixed while the vertex scale $|z|\asymp \e^{-T/\alpha}$ shrinks.  On that scale, multiplication of the \emph{full} error of the infinite shifted rule by $\e^T$ removes every occurrence of $T$.  In particular, the earlier left-truncation error is not discarded or treated as lower order: it becomes one $T$-independent component of the profile.  Thus \eqref{eq:profile} is an exact rescaling identity, not a formal local expansion.  The finite LP in Theorem~\ref{thm:exact} only pushes the later right-tail and compression errors to $o(\e^{-T})$.  The core-only realization keeps the first $N$ prescribed poles; the hybrid realization also retains a short, explicitly counted transition band to avoid an unnecessarily large polynomial part on wide sectors.

\begin{remark}[What the theorem does not assert]\label{rem:scope}
The balanced spacing maximizes a proved upper-bound exponent, but we prove neither a matching sector minimax lower bound nor an exact asymptotic for the best approximant in each shifted LP space.  Theorem~\ref{thm:exact} instead concerns one explicit dense/balanced sequence.  For $\Theta=0$, $g=1$, and $m=0$, $c_0(\sigma_*)=2\pi\sqrt\alpha$ agrees with Stahl's minimax exponent~\cite{Stahl03}.  The interval, the two-ray set, and every radial set with angles in $[-\Theta,\Theta]$ follow by restriction from the same sector estimate.  An analytic $g$ creates no new singularity, though a zero at the origin can weaken the leading one; the stated upper bound still applies.
\end{remark}

Before entering the proof, it helps to keep the construction in four pieces.  We first truncate the integral on the left, losing $O((1+T)^m\e^{-T})$.  We then apply the shifted quadrature rule; this is the step that creates the angle-dependent exponential rate.  Finally, to obtain a finite LP representation, we truncate the far-right quadrature tail and compress sufficiently remote poles into a polynomial.  The proof is arranged in exactly this order so that one can see which estimate is responsible for which part of the final error.

The rest of the paper is organized as follows. Section~\ref{sec:quadrature} proves the shifted quadrature estimate once and for all.  Section~\ref{sec:bridge} turns the infinite quadrature into a finite LP without changing its leading exponent and then proves Theorem~\ref{thm:main}.  The exact-asymptotic argument revisits the same construction but makes the final truncation and compression errors smaller than the natural $\e^{-T}$ vertex scale.  Section~\ref{sec:numerics} tests the resulting predictions while keeping numerical evidence separate from minimax or stability claims, and conclusions are given in Section~\ref{sec:conclusions}. 

\section{One shifted quadrature estimate}\label{sec:quadrature}

This section contains the only genuinely delicate quadrature estimate in the paper.  The strategy is simple to state.  We leave the integrable $u^{-1/2}$ endpoint factor alone when applying the one-sided Poisson formula. Inside each Fourier mode we then put $u=w^2$; the square-root singularity disappears, and all moving poles line up on one vertical line in the $w$-plane.  The contour can then be pushed deep into the half-plane except for a short notch around that pole line.  Everything needed later is summarized in Theorem~\ref{thm:quad}; readers mainly interested in the final LP construction may treat it as the analytic input for Section~\ref{sec:bridge}. Its loss-free $m=0$ bound will also give the global profile bound needed for Theorem~\ref{thm:exact}.

\subsection{The integral and its fixed polynomial weight}
We begin with an integral representation in which the logarithmic power is carried by a fixed polynomial weight.  This keeps the later quadrature argument identical for every fixed $m$.

Define
\begin{equation}\label{eq:Lambda}
A_j(\alpha)=\frac{d^j}{d\alpha^j}\frac{\sin(\pi\alpha)}\pi,
 \qquad
 \Lambda_{m,\alpha}(s)=\sum_{j=0}^m\binom mj
       \frac{A_{m-j}(\alpha)}{\alpha^{j+1}}s^j.
\end{equation}
In particular, $\Lambda_{0,\alpha}=\sin(\pi\alpha)/(\pi\alpha)$ and $\Lambda_{1,\alpha}(s)=\cos(\pi\alpha)/\alpha+ s\sin(\pi\alpha)/(\pi\alpha^2)$.  Differentiating the Stieltjes formula at fixed integration variable, and only then setting $t=\e^{s/\alpha}$, gives
\begin{equation}\label{eq:integral}
z^\alpha(\log z)^m =\int_{-\infty}^\infty
    \Lambda_{m,\alpha}(s)\frac{z\e^s}{\e^{s/\alpha}+z}\,ds.
\end{equation}
For completeness, differentiation of $z^\alpha=(\sin(\pi\alpha)/\pi) \int_0^\infty zt^{\alpha-1}/(t+z)\,dt$ is justified on compact subintervals $[\alpha_-,\alpha_+]\Subset(0,1)$ by the integrable majorant $C(t^{\alpha_- -1}+t^{\alpha_+ -1})(1+|\log t|)^m/(1+t)$, locally uniformly in the slit $z$-plane.  This yields~\eqref{eq:Lambda} after the substitution.  In particular, it does not mean differentiating an $\alpha$-dependent set of rational poles.

Put $\kappa=\alpha/(1-\alpha)$ and $d_\Theta=\cos(\Theta/2)>0$. The elementary inequality
\begin{equation}\label{eq:sector-separation}
|a+r\e^{\ii\theta}|\ge d_\Theta(a+r),
 \qquad a,r\ge0,\quad |\theta|\le\Theta,
\end{equation}
follows by squaring both sides.  It gives
\begin{equation}\label{eq:real-envelope}
 \left|\frac{z\e^s}{\e^{s/\alpha}+z}\right|
 \le d_\Theta^{-1}\min\{\e^s,\ |z|\e^{-s/\kappa}\}.
\end{equation}
Thus~\eqref{eq:integral} extends continuously to $z=0$, and its left tail below $-T$ is at most $C(1+T)^m\e^{-T}$, uniformly on $\Sig$.

It is convenient to prove the quadrature estimate for any fixed polynomial $P$ of degree at most $m$, subsequently taking $P=\Lambda_{m,\alpha}$. For $T>0$ define
\begin{equation}\label{eq:Phi}
 \begin{aligned}
 \Phi_T(w;z)&=\frac{z\e^{w-T}}{\e^{(w-T)/\alpha}+z},
 &\Psi_T(w;z)&=P(w-T)\Phi_T(w;z),\\
 F_T(u;z)&=\frac{\Psi_T(\sqrt u;z)}{2\sqrt u},&u>0.
 \end{aligned}
\end{equation}
The omitted left tail is the only difference between~\eqref{eq:integral} and $\int_0^\infty F_T(u;z)\,du$ when $P=\Lambda_{m,\alpha}$. The transformed square-root endpoint is integrable but in general not in $L^2$.  No replacement of $F_T$ at that endpoint will be made.

\subsection{An Abel identity and the shift kernel}
The shift enters only through a one-sided Poisson formula.  We state it in a form that applies directly to the half-line integrand and does not require an artificial value at the endpoint.

The point of the Abel regularization is not cosmetic.  Because the transformed integrand is only $O(u^{-1/2})$ at the endpoint, its Fourier series need not be absolutely summable, so a direct interchange of the mode sum with the later contour deformation would be unjustified.  We therefore damp the Fourier modes first, perform every deformation at fixed damping parameter, and remove the damping only after uniform contour bounds are available.  The interior shift is essential here: the sampling point $(1-\delta)h$ stays away from the endpoint of the zero extension when $\delta$ ranges in a compact subset of $(0,1)$.

For a half-line function $f$ write
\begin{equation}\label{eq:Qhat}
Q_{h,\delta}(f)=h\sum_{k=1}^\infty f((k-\delta)h),\qquad
 \widehat f(\xi)=\int_0^\infty f(u)\e^{-2\pi\ii\xi u}\,du.
\end{equation}
\begin{lemma}[One-sided shifted Poisson formula]\label{lem:Abel}
Let $f\in L^1(0,\infty)$, $h>0$, and $0<\delta<1$.  Assume that $Q_{h,\delta}(f)$ converges absolutely and that the periodization of the zero extension of $f$ has a representative continuous at $(1-\delta)h$, with value equal to the sampled series there.  Then
\begin{equation}\label{eq:Abel}
Q_{h,\delta}(f)-\int_0^\infty f(u)\,du =\lim_{\rho\uparrow1}\sum_{n\ne0}
    \rho^{|n|}\e^{-2\pi\ii n\delta}\widehat f(n/h).
\end{equation}
\end{lemma}
\begin{proof}
Let $f_0$ be the zero extension and $\mathcal P_hf(t)=\sum_{k\in\mathbb Z}f_0(t+kh)$. Tonelli's theorem gives $\int_0^h\sum_k|f_0(t+kh)|\,dt=\norm{f}_1$. For $\rho<1$, the bound $|\widehat f|\le\norm f_1$ permits summation under the integral.  With the circular Poisson kernel $P_\rho(v)=(1-\rho^2)/(1-2\rho\cos v+\rho^2)$, the sum including $n=0$ is
\[
 \int_0^h\mathcal P_h f(t)
P_\rho\!\left(\frac{2\pi}{h}(t-(1-\delta)h)\right)dt.
\]
The kernel has mass $h$ and tends uniformly to zero off any neighborhood of $(1-\delta)h$.  Continuity at that point and integrability of the periodization therefore give the limit $Q_{h,\delta}(f)$. Subtract the $n=0$ term.
\end{proof}

For~\eqref{eq:Phi} the lemma's assumptions can be checked directly.  The endpoint estimate $F_T(u;z)=O(u^{-1/2})$ gives integrability at $0$, while the large-$u$ tail is stretched-exponentially summable.  Hence the periodized tail converges locally uniformly away from the lattice endpoints.  Since $(1-\delta)h$ remains in a compact subinterval of $(0,h)$ for $\delta\in D$, the periodization is continuous at the sampled point and its value there is exactly the absolutely convergent node sum required in Lemma~\ref{lem:Abel}.  Substitution $u=w^2$ in each Fourier transform then removes the singular factor:
\begin{equation}\label{eq:Fourier-w}
 \widehat F_T(n/h;z)=\int_0^\infty
       \Psi_T(w;z)\e^{-2\pi\ii n w^2/h}\,dw.
\end{equation}
The positive modes will be deformed into the lower half-plane and the negative modes into the upper half-plane.  Their geometric kernels are
\begin{equation}\label{eq:kernels}
 \begin{aligned}
 K_{h,\delta}^-(w)&=\frac{q_-(w)}{1-q_-(w)},&
 q_-(w)&=\e^{-2\pi\ii(\delta+w^2/h)},\\
 K_{h,\delta}^+(w)&=\frac{q_+(w)}{1-q_+(w)},&
 q_+(w)&=\e^{ 2\pi\ii(\delta+w^2/h)}.
 \end{aligned}
\end{equation}
For $w=x-\ii y$, $x,y>0$, and $b=4\pi/h$,
\begin{equation}\label{eq:kernel-envelope}
|K_{h,\delta}^-(w)|\le\frac{\e^{-bxy}}{1-\e^{-bxy}}.
\end{equation}
The sign-reversed statement holds in the upper half-plane.

\begin{proposition}[A kernel (not error) optimum]\label{prop:shift}
For every $0<\delta<1$,
\begin{equation}\label{eq:Adelta}
 \sup_{t\ge0}|K_{h,\delta}^-(t(1-\ii))|
=\sup_{t\ge0}|K_{h,\delta}^+(t(1+\ii))| =A_\delta:=\frac1{2\sin(\pi\delta)}.
\end{equation}
The midpoint $\delta=1/2$ uniquely minimizes this quantity.
\end{proposition}
\begin{proof}
On the lower diagonal, $q_- =v\e^{-2\pi\ii\delta}$ with $v=\e^{-bt^2}\in(0,1]$.  The derivative of $v^2/(1+v^2-2v\cos(2\pi\delta))$ has the sign of $1-v\cos(2\pi\delta)\ge0$.  Its maximum occurs at $v=1$ and equals $[2\sin(\pi\delta)]^{-2}$.  The upper kernel is its conjugate. The same bound holds with $v$ replaced by $\rho v$, $0<\rho<1$.
\end{proof}

The phase gap in~\eqref{eq:Adelta} explains why merely omitting the sample at $u=0$ is not the same as shifting inside the period cell.  For the endpoint lattice $\delta=0$, $K_{h,0}^-(t(1-\ii))\sim h/(4\pi t^2)$ as $t\downarrow0$. This is a degeneration of the present kernel representation, not a symmetric divergence of the actual quadrature at both ends of the shift interval.  At fixed $T,h,z\ne0$,
\begin{equation}\label{eq:asymmetric-endpoints}
 \begin{aligned}
 Q_{h,\delta}(F_T)&\longrightarrow h\sum_{k\ge1}F_T(kh)
       &&(\delta\downarrow0),\\
 hF_T((1-\delta)h)&=\frac{\Psi_T(0;z)\sqrt h}{2\sqrt{1-\delta}}+O(1)
       &&(\delta\uparrow1).
 \end{aligned}
\end{equation}
The first limit follows by dominated convergence of the node sum and the second by expanding $\Psi_T$ at zero; only the latter generically diverges.  The actual LP error also combines quadrature and truncation terms with complex phases, so Proposition~\ref{prop:shift} neither orders these errors nor identifies a universally best finite-$N$ shift.

\subsection{A localized notch removes artificial prefactors}
A global horizontal contour would pass too close to the moving pole line and would introduce powers of $T$ that are not part of the true leading scale. The remedy is local: keep the deep contour everywhere except for a narrow notch around the nearest pole.  The next theorem shows that this recovers the full nearest-pole exponent.

For $z=r\e^{\ii\theta}\ne0$ put
\begin{equation}\label{eq:pole-geometry}
X=T+\alpha\log r\le T,\qquad a_-(\theta)=\alpha(\pi-\theta),\quad a_+(\theta)=\alpha(\pi+\theta).
\end{equation}
The poles of $\Phi_T$ are $X+\ii\alpha(\theta+(2\ell+1)\pi)$, $\ell\in\mathbb Z$; all are on the vertical line $\Re w=X$. Let us set
\begin{equation}\label{eq:q0}
q_0=\frac{4\pi\alpha(\pi-\Theta)}h.
\end{equation}

\begin{theorem}[Uniform shifted quadrature bound]\label{thm:quad}
Fix $h>0$, $D\Subset(0,1)$, and a polynomial $P$ of degree at most $m$. For $T\ge1$ and $F_T$ from~\eqref{eq:Phi}, write $\mathcal E_{h,\delta}(T,z)=Q_{h,\delta}(F_T(\cdot;z))- \int_0^\infty F_T(u;z)\,du$.  Then
\begin{equation}\label{eq:quad-sharp}
 \sup_{\delta\in D}\sup_{z\in\Sig}
 \left|Q_{h,\delta}(F_T(\cdot;z))-\int_0^\infty F_T(u;z)\,du\right|
 \le C\begin{cases}
 (1+T)^m\e^{-T},&q_0\ge1,\\
 \e^{-q_0T},&0<q_0<1.
 \end{cases}
\end{equation}
In the second line the constant may deteriorate as $q_0\uparrow1$.
\end{theorem}

\begin{proof}
All constants in the proof may depend on the fixed parameters, but not on $T,z,\delta$.  Write $a_0=\alpha(\pi-\Theta)>0$, $a_1=\alpha(\pi+\Theta)$, $b=4\pi/h$, and choose
\begin{equation}\label{eq:Xstar}
X_*\ge\max\{4,\,2/a_0,\,2/\alpha,\,4(a_1+2)\}.
\end{equation}
The numerical constants in this threshold are inessential: they make the notch positive, place its left horizontal segment before the pole line, and make the angles below small.  The proof has one bookkeeping principle.  In the large-$X$ region considered below we have $X>X_*>0$.  With $a=a_-(\theta)$ below and $a=a_+(\theta)$ above, the near-pole estimate has the form
\[
   \e^{-T+(1-q)X},\qquad q=ba\ge q_0.
\]
Since $X>0$, if $q_0\ge1$ then $q\ge1$ and this is at most $\e^{-T}$.  If $q_0<1$, then, with $s=T-X\ge0$,
\[
   \e^{-T+(1-q)X}
   \le \e^{-T+(1-q_0)X}
   =\e^{-q_0T}\,\e^{-(1-q_0)s}.
\]
Thus the two branches of~\eqref{eq:quad-sharp} come from the same contour geometry; the rest of the proof is devoted to showing that no contour piece loses this exponent.

\emph{The bounded-$X$ region.} For real $w\ge0$,~\eqref{eq:sector-separation} gives
\begin{equation}\label{eq:small-X}
 \begin{aligned}
 |\Psi_T(w;z)|&\le
 C(1+T+w)^m\frac{\e^{w-T}}{1+\e^{(w-X)/\alpha}}\\
 &\le C\e^{X_*/\alpha}(1+T+w)^m\e^{-T-w/\kappa},
 \qquad X\le X_*.
 \end{aligned}
\end{equation}
The integral in $w$, and the sum with factors $h/(2\sqrt{(k-\delta)h})$, are bounded by $C(1+T)^m\e^{-T}$, uniformly on $D$. If $q_0<1$, this is at most $C\e^{-q_0T}$.  This also covers $z=0$ by continuity.  Moreover, since $X=T+\alpha\log r\le T$ for $0<r\le1$, if $T\le X_*$ then $X\le X_*$ throughout the sector; in that case the bounded-$X$ estimate already covers the whole problem.

\emph{The contour for $X>X_*$.} We describe the lower contour; the upper one is obtained by reversing imaginary signs and using $a=a_+(\theta)$ instead of $a=a_-(\theta)$. With
\begin{equation}\label{eq:notch-height}
a=a_-(\theta),\qquad \eta=a-X^{-1},\qquad B=\eta+1,
\end{equation}
take the diagonal $w=t(1-\ii)$, $0\le t\le B$, followed by the graph $w=x-\ii y(x)$, $x\ge B$, where
\begin{equation}\label{eq:notch-graph}
y(x)=\begin{cases}
 B,&B\le x\le X-1, \mbox{ or }x\ge X+1\\
 \eta+|x-X|,&|x-X|\le1.
 \end{cases}
\end{equation}
Condition~\eqref{eq:Xstar} ensures $B<X-1$ and $\eta\ge a_0/2$.  At $x=X$ the contour is above the nearest lower pole, since $y(X)=a-X^{-1}<a$.  All other lower poles are deeper on the same vertical line.  Hence the region between the real half-line and this contour contains no pole, even though its horizontal portions lie below the nearest-pole height.  See Figure~\ref{fig:contour}.

Apply Lemma~\ref{lem:Abel} and~\eqref{eq:Fourier-w}.  Write $\mathcal C_-$ and $\mathcal C_+$ for the lower and upper contours, each oriented from $0$ toward infinity. At fixed $0<\rho<1$, deform each positive Fourier mode through $\mathcal C_-$ and each negative mode through $\mathcal C_+$.  On a closing segment $w=R-\ii y$, $0\le y\le B$, the positive-mode Gaussian has modulus $\exp(-4\pi nRy/h)\le1$, while $|\Psi_T(R-\ii y;z)|\le C(1+T+R)^m\exp(-(R-T)/\kappa)$ for large $R$; hence the far vertical side vanishes.  The upper closing side is identical after reflection.  Before $\rho\uparrow1$, the Abel-damped mode series is absolutely summable and may therefore be summed under the contour integrals.  Moreover the damped geometric kernels satisfy the same bounds as the undamped ones: Proposition~\ref{prop:shift} controls the diagonal uniformly in $\rho$ and $\delta\in D$, while~\eqref{eq:kernel-envelope} controls every graph and horizontal piece.  The piecewise-contour corners have measure zero, and the explicit estimates below give an integrable majorant independent of $\rho$, $\delta\in D$, and $z\in\Sig$.  Dominated convergence can therefore remove the Abel factor and gives the exact identity
\begin{equation}\label{eq:contour-identity}
 \mathcal E_{h,\delta}(T,z)
=\int_{\mathcal C_-}\Psi_T(w;z)K^-_{h,\delta}(w)\,dw +\int_{\mathcal C_+}\Psi_T(w;z)K^+_{h,\delta}(w)\,dw.
\end{equation}
The modes are deformed before summation; the summed kernels have sampling poles on the original real axis and cannot be deformed there directly.

\emph{The two sloping sides of the notch.} Put $v=x-X$, $|v|\le1$, and $s=T-X\ge0$.  On the notch,
\begin{equation}\label{eq:notch-product}
xy=(X+v)(a-X^{-1}+|v|)
 \ge aX-1+\tfrac12X|v|.
\end{equation}
The denominator of $\Phi_T$ after cancellation of its $z$ factor is
\begin{equation}\label{eq:notch-denom}
1-\exp\!\left(\frac{v+\ii(X^{-1}-|v|)}\alpha\right),
 \;\,
 \left|1-\exp\!\left(\frac{v+\ii(X^{-1}-|v|)}\alpha\right)\right|
 \ge c(\alpha)(|v|+X^{-1}).
\end{equation}
Here $c(\alpha)>0$ is independent of $X$, $v$, and $\theta$ in the stated ranges.  To see the inequality, use $|1-\e^{v/\alpha+\ii\phi}|\ge|1-\e^{v/\alpha}|\ge c(\alpha)|v|$ when $|v|\ge(2X)^{-1}$.  When $|v|<(2X)^{-1}$, the imaginary argument lies between $(2\alpha X)^{-1}$ and $(\alpha X)^{-1}\le1/2$; the imaginary part gives a lower bound $c(\alpha)X^{-1}$. Also $|P(w-T)|\le C(1+s)^m$ and $|dw|\le\sqrt2\,dv$. Equations~\eqref{eq:kernel-envelope}--\eqref{eq:notch-denom} therefore bound the notch contribution by
\begin{align}\label{eq:notch-integral}
C(1+s)^m\e^{X-T-baX}
 \int_{-1}^1\frac{\e^{-bX|v|/2}}{|v|+X^{-1}}\,dv
 &\le C(1+s)^m\e^{-T+(1-ba)X}.
\end{align}
Indeed the substitution $y=X|v|$ bounds the integral by $2\int_0^\infty\e^{-by/2}/(1+y)\,dy$. The apparent $O(X)$ denominator amplification is thus localized to an $O(X^{-1})$-width region.  No factor $X$ or $\log X$ is lost.

\emph{The remaining contour pieces.} On both horizontal pieces $|x-X|\ge1$.  The reverse triangle inequality implies
\begin{equation}\label{eq:off-resonance}
|1+\e^{(w-X)/\alpha-\ii\theta}|
 \ge c(\alpha)(1+\e^{(x-X)/\alpha}).
\end{equation}
Let $q=ba$ and $q_B=bB$; then $q\ge q_0$ and $q_B-q\ge b/2$. On the left horizontal piece the integral is at most
\begin{equation}\label{eq:left-horizontal}
C\e^{-T}\int_0^X(1+T-x)^m\e^{(1-q_B)x}\,dx.
\end{equation}
If $q_0\ge1$, this is bounded by $C(1+T)^m\e^{-T}$, since $q_B\ge1+b/2$.  If $q_0<1$, rewrite its integrand, including $\e^{-T}$, as
\[
 \e^{-q_0T}(1+T-x)^m
 \e^{-(1-q_0)(T-x)}\e^{-(q_B-q_0)x}.
\]
The polynomial times the first exponential is bounded for $T-x\ge0$, and the last exponential is at most $\e^{-bx/2}$.  This gives $C\e^{-q_0T}$ without an algebraic prefactor.

On the right horizontal piece put $x=X+v$, $v\ge1$. The logistic tail and~\eqref{eq:kernel-envelope} lead to
\begin{equation}\label{eq:right-horizontal}
C(1+s)^m\e^{-T+(1-q)X}
       \int_1^\infty(1+v)^m\e^{-(q_B+1/\kappa)v}\,dv
 \le C(1+s)^m\e^{-T+(1-q)X}.
\end{equation}
The diagonal has bounded length, $t\le B<X-1$, and contributes at most $C A_D(1+T)^m\e^{-T}$, where $A_D=\sup_{\delta\in D}A_\delta<\infty$. Finally, recall that this part of the proof has $X>X_*>0$.  If $q_0\ge1$, then $q\ge1$ and the right sides of \eqref{eq:notch-integral} and~\eqref{eq:right-horizontal} are bounded by $C(1+T)^m\e^{-T}$.  If $q_0<1$, the same comparison $q\ge q_0$ and $X>0$, followed by $s=T-X$, gives
\[
C(1+s)^m\e^{-T+(1-q)X}
\le C\e^{-q_0T}(1+s)^m\e^{-(1-q_0)s}
\le C\e^{-q_0T}.
\]
For the upper contour, $a=a_+(\theta)=\alpha(\pi+\theta)\ge\alpha(\pi-\Theta)$, so again $q=ba\ge q_0$ and every preceding estimate applies after reflection.  Combining all pieces proves~\eqref{eq:quad-sharp}.
\end{proof}

\begin{figure}[!tbp]
 \centering
 \includegraphics[width=.82\linewidth]{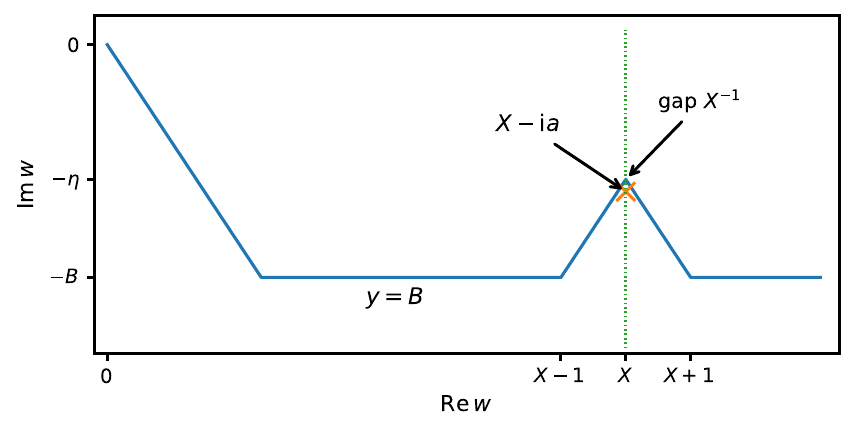}
 \caption{Schematic lower contour, oriented from $0$ to $+\infty$ along the real half-line before deformation.  All integrand poles have real part $X$.  The contour reaches height $B$, follows the horizontal piece to $X-1$, makes its local notch across $[X-1,X+1]$, and then returns to height $B$.  The notch passes above the nearest pole at $X-\ii a$; its vertical gap is $X^{-1}$ and the effective integration width near that pole is also $O(X^{-1})$.}
 \label{fig:contour}
\end{figure}

\begin{remark}[Why a global horizontal bound is wasteful]\label{rem:prefactor}
Keeping an entire horizontal contour at height $a-O(T^{-1})$ and bounding its denominator by $O(T)$ loses a power of $T$ before integration; at the balanced exponent an additional interval-length factor can appear. The notch avoids both losses.  The factor $(1+T)^m$ in the first line of~\eqref{eq:quad-sharp} comes from the logarithmic weight near the truncation endpoint, not from proximity to a pole.  In the second line, $(1+T-X)^m$ is absorbed by $\e^{-(1-q_0)(T-X)}$. Thus the distinction between the two prefactors is a consequence of the same proof, not a separate argument for each logarithmic power.
\end{remark}

\section{The finite LP construction}\label{sec:bridge}

The quadrature estimate already contains the difficult part of the analysis.  What remains is algebraic bookkeeping: identify the quadrature poles, discard a sufficiently small right tail, and replace the poles that are far from the sector by a polynomial.  No additional contour argument is required.  It is useful to keep the roles of the parameters visible throughout this section: $h$ fixes the node spacing, $T$ is the natural root-exponential scale, and $M$ is the right-tail cutoff.

Set
\begin{align}\label{eq:discretization}
 h&=\alpha^2\sigma^2, \qquad T=\sqrt{Nh}, \qquad u_k=(k-\delta)h,\\
 w_k&=\sqrt{u_k}, \qquad
 M=\left\lceil\frac{N}{(1-\alpha)^2}+\delta-1\right\rceil.\nonumber
\end{align}
Then $u_N<T^2<u_{N+1}$ and $w_{M+1}-T\ge\kappa T$.  Thus the index $N$ separates the prescribed core poles from the poles lying to the left of $-1$, while $M$ is far enough to make the omitted quadrature tail exponentially smaller than the active approximation scale.  For every $k\ge1$ write
\begin{equation}\label{eq:weights}
a_k=\e^{(w_k-T)/\alpha},\qquad b_k=\frac{h}{2w_k}\Lambda_{m,\alpha}(w_k-T)\e^{w_k-T}.
\end{equation}
The finite quadrature is simply
\begin{equation}\label{eq:finite-Q}
S_{N,\delta}(z)=\sum_{k=1}^M b_k\frac{z}{z+a_k}.
\end{equation}
For $k\le N$, $-a_k$ is the prescribed pole location~\eqref{eq:poles} and lies in $(-1,0)$; for $k>N$ it lies strictly to the left of $-1$.

\begin{lemma}[Tails and an exact polynomial remainder]\label{lem:bridge}
Let $P=\Lambda_{m,\alpha}$ in~\eqref{eq:Phi}.  Uniformly for $\delta\in D$,
\begin{equation}\label{eq:tails}
 \sup_{z\in\Sig}\left|z^\alpha(\log z)^m-
          \int_0^\infty F_T(u;z)\,du\right|
+\sup_{z\in\Sig}h\sum_{k>M}|F_T(u_k;z)|
 \le C(1+T)^m\e^{-T}.
\end{equation}
Put
\begin{equation}\label{eq:Rtheta}
R_\Theta=\max\{1,2\sin(\Theta/2)\}<2,\qquad H_N(z)=\sum_{k=N+1}^M b_k\frac{z}{z+a_k}.
\end{equation}
If $P_{N,n}$ is the degree-$n$ Taylor polynomial of $H_N$ about $z=1$, then for every $n\ge0$,
\begin{equation}\label{eq:compression}
 \norm{H_N-P_{N,n}}_{\Sig}
 \le C\left({R_\Theta}/{2}\right)^{n+1}.
\end{equation}
Thus the choice
\begin{equation}\label{eq:degree}
n_N=\left\lceil\frac{T}{\log(2/R_\Theta)}\right\rceil
\end{equation}
gives polynomial-compression error $O(\e^{-T})$ and degree $O(\sqrt N)$.
\end{lemma}
\begin{proof}
The left tail was bounded after~\eqref{eq:real-envelope}. For $k>M$ put $s_k=w_k-T\ge\kappa T$. The right tail is bounded by
\[
C h\sum_{k>M}\frac{(1+s_k)^m}{2w_k}\e^{-s_k/\kappa}.
\]
As a function of $s$, $(1+s)^m\e^{-s/\kappa}/(T+s)$ is decreasing on $[\kappa T,\infty)$ for all sufficiently large $T$, as follows by logarithmic differentiation.  A decreasing-function sum--integral comparison, using $ds=du/(2\sqrt u)$, bounds the sum by its first term plus the integral $C\int_{\kappa T}^\infty(1+s)^m\e^{-s/\kappa}\,ds$. Both are $O((1+T)^m\e^{-T})$, uniformly in $\delta\in D$. The remaining bounded range of $T$ is handled by the same tail majorant and the positive lower bound on the first node.

To prove~\eqref{eq:compression}, first observe that $|z-1|\le R_\Theta$ on $\Sig$.  This follows by maximizing the convex quadratic $1+r^2-2r\cos\theta$ over $0\le r\le1$. For $a>1$, the exact remainder of the Taylor polynomial $\tau_{n,a}$ of $z/(z+a)$ at $1$ is
\begin{equation}\label{eq:exact-remainder}
 \frac{z}{z+a}-\tau_{n,a}(z)
=-\frac{a}{z+a}\left(-\frac{z-1}{1+a}\right)^{n+1}.
\end{equation}
By~\eqref{eq:sector-separation}, $|a/(z+a)|\le d_\Theta^{-1}$. Since $2/(1+a)\le1$ and $n+1\ge1$,
\[
 \norm{H_N-P_{N,n}}_{\Sig}
 \le C\left(\frac{R_\Theta}{2}\right)^{n+1}
       \sum_{k>N}|b_k|\frac{2}{1+a_k}.
\]
Here $s_k=w_k-T>0$ and
\[
 \sum_{k>N}|b_k|\frac{2}{1+a_k}
 \le C h\sum_{k>N}\frac{(1+s_k)^m}{2w_k}\e^{-s_k/\kappa}
 \le C.
\]
For the last inequality, bound $(1+s)^m\e^{-s/\kappa}\le C\e^{-s/(2\kappa)}$ and compare the resulting decreasing sum with its integral starting at $u=T^2$. This also proves uniformity in $N$ and $\delta$.  Equation~\eqref{eq:degree} now implies the stated compression error.
\end{proof}

The exact remainder~\eqref{eq:exact-remainder} avoids the auxiliary choice $R_\Theta<R_0<2$, replacing the more conservative degree $\lceil T/\log(R_0/R_\Theta)\rceil$ by~\eqref{eq:degree}. For $\Theta\uparrow\pi$ this is still an expensive existence construction: $\log(2/R_\Theta)\sim(\pi-\Theta)^2/8$.  At balanced spacing the degree coefficient therefore grows as $(\pi-\Theta)^{-3/2}$.  Nothing in the fixed-angle $O(\sqrt N)$ statement makes that coefficient small.

The large polynomial degree above is a consequence of where the rational--polynomial split is made, not of the tapered core itself.  Instead of replacing every pole as soon as it passes $-1$, one can retain a short transition band and compress only the poles that are genuinely far from the unit disk.  These added poles are representation/compression poles outside the unit disk, not a second cluster intended to resolve the branch point.  There is one further economy in the sparse branch.  Put
\begin{equation}\label{eq:lambda-sigma}
 \lambda_\sigma:=\min\{1,q_0\}
 =\min\left\{1,{\sigma_*^2}/{\sigma^2}\right\}.
\end{equation}
The target scale in Theorem~\ref{thm:quad} is $e^{-\lambda_\sigma T}$ up to the fixed logarithmic prefactor in the dense/balanced branch.  If the cutoff is $A=\e^y$, moving it from $1$ to $A$ retains about $(2\alpha T/h)y$ additional poles, while it is enough to approximate the remaining geometric tail to degree about $\lambda_\sigma T/y$.  The optimal split therefore balances these two costs at the accuracy actually required by the active branch.

\begin{proposition}[Rate-matched dimension-controlled hybrid compression]\label{prop:dimension}
Fix $m\in\N_0$ and $A>1$, put $y=\log A$, and consider the singular prototype $g=1$.  For all sufficiently large $N$ define
\begin{equation}\label{eq:hybrid-cut}
K_A=\left\lfloor \delta+\frac{(T+\alpha y)^2}{h}\right\rfloor,
 \qquad n_{A,\lambda}=\left\lceil\frac{\lambda_\sigma T}{y}\right\rceil.
\end{equation}
Retain the quadrature terms with $k\le K_A$ and replace the terms $K_A<k\le M$ by their degree-$n_{A,\lambda}$ geometric expansion about $z=0$. Then the first $N$ prescribed pole locations remain those in \eqref{eq:poles}, the additional $K_A-N$ locations lie in $[-A,-1)$, and the resulting LP function has the same branchwise error bound \eqref{eq:main-bound} as the core-only construction. Its total dimension satisfies, uniformly for $\delta\in D$,
\begin{equation}\label{eq:hybrid-dimension}
L_A=K_A+n_{A,\lambda}+1
 =N+T\left(\frac{2\alpha}{h}y+\frac{\lambda_\sigma}{y}\right)+O(1).
\end{equation}
Within this one-parameter split, the leading correction is minimized by
\begin{equation}\label{eq:Astar}
 \log A_{*,\lambda}=\sqrt{\frac{\lambda_\sigma h}{2\alpha}},
 \qquad
L_{A_{*,\lambda}}=N+2\sqrt{2\alpha\lambda_\sigma N}+O(1).
\end{equation}
For $\sigma\le\sigma_*$ this reduces to $N+2\sqrt{2\alpha N}+O(1)$.  For $\sigma>\sigma_*$ it becomes
\begin{equation}\label{eq:sparse-dimension}
 L_{A_{*,\lambda}}=N+2\frac{\sigma_*}{\sigma}\sqrt{2\alpha N}+O(1).
\end{equation}
At balanced spacing, $\log A_{*,\lambda}=\sqrt{2\pi(\pi-\Theta)}$.
\end{proposition}
\begin{proof}
The condition $a_k\le A$ is equivalent to $w_k\le T+\alpha y$, which gives \eqref{eq:hybrid-cut}.  For fixed $A$, $N\le K_A<M$ once $N$ is large.  Thus $1<a_k\le A$ for $N<k\le K_A$.

For $a>1$ and $|z|\le1$ the exact degree-$n$ expansion at the origin is
\begin{equation}\label{eq:zero-remainder}
 \frac{z}{z+a}
=\sum_{j=1}^{n}(-1)^{j-1}\frac{z^j}{a^j} +(-1)^n\frac{z^{n+1}}{a^n(z+a)}.
\end{equation}
Hence, for $a>A$,
\[
 \left|\frac{z}{z+a}-\sum_{j=1}^{n}(-1)^{j-1}\frac{z^j}{a^j}\right|
 \le \frac{A^{-n}}{a-1}.
\]
For $k>K_A$ write $s_k=w_k-T>\alpha y$.  Using $|\Lambda_{m,\alpha}(s)|\le C(1+s)^m$ and
\[
 \frac{\e^{s_k}}{\e^{s_k/\alpha}-1}
 \le C_A\e^{-s_k/\kappa},
\]
the same sum--integral comparison as in Lemma~\ref{lem:bridge} gives
\begin{equation}\label{eq:hybrid-weight-bound}
 \sum_{k>K_A}\frac{|b_k|}{a_k-1}
 \le C_A h\sum_{k>K_A}
 \frac{(1+s_k)^m}{2w_k}\e^{-s_k/\kappa}\le C_A.
\end{equation}
Thus the far-pole compression error is at most
$C_AA^{-n_{A,\lambda}}\le C_A\e^{-\lambda_\sigma T}$.  If $q_0\ge1$, then $\lambda_\sigma=1$ and this is the $e^{-T}$ scale in the first line of Theorem~\ref{thm:quad}.  If $q_0<1$, then $\lambda_\sigma=q_0$ and the compression is matched directly to the $e^{-q_0T}$ scale in the second line; the tails in \eqref{eq:tails} are smaller and are absorbed there.  Together with Theorem~\ref{thm:quad}, this proves the same bound as before.

Finally,
\[
K_A-N=\frac{2\alpha T}{h}y+O(1),\qquad
n_{A,\lambda}=\frac{\lambda_\sigma T}{y}+O(1),
\]
which gives \eqref{eq:hybrid-dimension}.  Minimizing $2\alpha y/h+\lambda_\sigma/y$ over $y>0$ yields \eqref{eq:Astar}.  Since $T=\sqrt{Nh}$, each of the two leading corrections at the minimizer equals $\sqrt{2\alpha\lambda_\sigma N}$.  For $\sigma>\sigma_*$, \eqref{eq:lambda-sigma} gives $\sqrt{\lambda_\sigma}=\sigma_*/\sigma$, proving \eqref{eq:sparse-dimension}.  At $\sigma=\sigma_*$, $h=4\pi\alpha(\pi-\Theta)$ and the final formula for $A_{*,\lambda}$ follows.
\end{proof}

\begin{remark}[What the dimension formula measures]\label{rem:dimension}
The optimization in Proposition~\ref{prop:dimension} is only over this explicit transition-pole split; it is not a lower bound for all LP spaces. It also keeps the prescribed first $N$ tapered poles unchanged, so the additional rational functions resolve only the smooth transition from $|\xi|=1$ to $|\xi|=A_{*,\lambda}$.  In the sparse branch the smaller coefficient in \eqref{eq:sparse-dimension} comes from asking the polynomial compression to reach the proved sparse error scale, not from a new lower bound or a claim of finite-$N$ optimality.  For a general analytic multiplier $g$, the separate polynomial approximation in \eqref{eq:g-approx} contributes its own $O(\sqrt N)$ degree, so the explicit coefficient in \eqref{eq:Astar} is stated for the singular prototypes.
\end{remark}

Proposition~\ref{prop:dimension} optimizes the size of the final representation, not the preprocessing work used to assemble its coefficients. For the straightforward implementation above, $M=N/(1-\alpha)^2+O(1)$ quadrature terms are generated, and accumulating the $n_{A,\lambda}$ polynomial coefficients from the $M-K_A$ far-pole terms costs $O((M-K_A)n_{A,\lambda})$ arithmetic operations.  Once those coefficients have been formed, a value of the hybrid LP requires only $O(K_A+n_{A,\lambda})$ basis evaluations.  Operationally, the construction is: compute $h,T$ and $\lambda_\sigma$, choose $y=\sqrt{\lambda_\sigma h/(2\alpha)}$, retain the terms through $K_A$, and accumulate the remaining terms into the degree-$n_{A,\lambda}$ polynomial using \eqref{eq:zero-remainder}.  No claim of optimal coefficient-generation complexity or floating-point stability is made here; the dimension result concerns the stored and evaluated LP representation.

\subsection{Proof of the main theorem and analytic factors}
We first finish the singular prototype $g=1$.  The general analytic multiplier is then added by ordinary polynomial approximation; no new pole geometry is needed.

For $g=1$ define
\begin{equation}\label{eq:constructive-LP}
r^0_{N,\delta}(z)=\sum_{k=1}^N b_k\frac{z}{z+a_k}+P_{N,n_N}(z).
\end{equation}
The identity $b_k z/(z+a_k)=b_k-b_ka_k/(z+a_k)$ puts this in form~\eqref{eq:LP}, with the constants $b_k$ absorbed in its polynomial. Theorem~\ref{thm:quad} and Lemma~\ref{lem:bridge} give~\eqref{eq:main-bound} for this prototype.  Indeed $q_0\ge1$ is equivalent to $\sigma\le\sigma_*$, and $q_0T=4\pi(\pi-\Theta)\sqrt N/\sigma$. When $q_0<1$, both $(1+T)^m\e^{-T}$ tails are absorbed in $C\e^{-q_0T}$.

For an analytic $g$, geometric polynomial approximation on the compact set $\Sig$, whose complement is connected, supplies polynomials $q_N$ of degree $O(\sqrt N)$ such that
\begin{equation}\label{eq:g-approx}
 \norm{g-q_N}_{\Sig}\le\e^{-\gamma\sqrt N},\qquad
 \sup_N\norm{q_N}_{\Sig}<\infty,
\end{equation}
where any fixed $\gamma>c_\Theta(\sigma)$ may be prescribed after increasing the degree constant; this is the standard Bernstein--Walsh geometric approximation theorem \cite[Chap.~V]{Walsh69}.  Analyticity on a fixed neighborhood gives geometric decay in the polynomial degree, and choosing that degree to be a sufficiently large constant multiple of $\sqrt N$ makes its decay faster than the rational-approximation scale $\e^{-c_\Theta(\sigma)\sqrt N}$.  This is the only non-elementary polynomial-approximation fact needed for the analytic multiplier, not used for the large-pole tail. With $f_{\rm sing}=z^\alpha(\log z)^m$, we have
\[
 \norm{g f_{\rm sing}-q_Nr^0_{N,\delta}}_{\Sig}
 \le\norm{g-q_N}_{\Sig}\norm{f_{\rm sing}}_{\Sig}
+\norm{q_N}_{\Sig}\norm{f_{\rm sing}-r^0_{N,\delta}}_{\Sig}.
\]
Multiplication introduces no new finite poles, because
\begin{equation}\label{eq:g-closure}
 \frac{q_N(z)}{z-\xi}=\frac{q_N(\xi)}{z-\xi}
+\frac{q_N(z)-q_N(\xi)}{z-\xi},
\end{equation}
and the last term is a polynomial of degree at most $\deg q_N-1$.  Hence the polynomial part of $q_Nr^0_{N,\delta}$ has degree at most $n_N+\deg q_N=O(\sqrt N)$, while no new finite poles are introduced.  All bounds remain uniform on $D$.  This proves the main estimate.  Maximizing the two branches in~\eqref{eq:rates} gives \eqref{eq:balanced}; replacing $N$ by $L_{N,\delta}=N+O(\sqrt N)$ does not change the leading exponent.\hfill$\square$

\begin{remark}[The range of $\alpha$]\label{rem:alpha}
For $\nu=\ell+\alpha>1$, $\ell\in\N$, multiplication by $z^\ell$ provides a valid result with the exponent associated with the fractional part $\alpha$.  It does not recover the stronger law with $\nu$ in place of $\alpha$.  Matching that law requires a different regularization of the integral at infinity, not just the analytic-multiplier step. We therefore retain $0<\alpha<1$ in the stated parameter law.  The broader exponent range and the multi-corner construction in~\cite{XXYZ26} remain outside the present scope.
\end{remark}

\subsection{Vertex-scale profile and exact asymptotics}\label{sec:profile}
We now ask a finer question than an upper bound: for the explicit algebraic prototype, what is the actual leading error of this construction?  The natural scale lives near the sector vertex.  After zooming to that scale, the large parameter disappears from the infinite quadrature error and leaves a fixed profile.  This is why an exact constant can be identified.

We prove Theorem~\ref{thm:exact} for $m=0$ and $g=1$.  There are two logically separate tasks.  First we identify the exact $T$-free profile of the \emph{full error} of the infinite shifted rule and show that its supremum is finite and nonzero.  This full error already includes the left truncation introduced before quadrature.  Only after that do we return to the finite LP and choose the right-tail and compression cutoffs so that their combined error is $o(\e^{-T})$; these later finite-realization errors therefore cannot alter the profile constant.  Keep $h=\alpha^2\sigma^2$ fixed and write
\begin{equation}\label{eq:infinite-Q}
 \mathcal Q_{T,h,\delta}(z)
=c_\alpha h\sum_{k=1}^\infty
 \frac{z\e^{w_k-T}}
{2w_k\bigl(\e^{(w_k-T)/\alpha}+z\bigr)},
 \qquad w_k=\sqrt{(k-\delta)h}.
\end{equation}
This is the infinite shifted rule before right-tail truncation and polynomial compression.  To keep the earlier left truncation visible, let
\[
 I_T(z)=c_\alpha\int_{-T}^{\infty}
 \frac{z\e^s}{\e^{s/\alpha}+z}\,ds.
\]
Then, identically,
\[
 z^\alpha-\mathcal Q_{T,h,\delta}
 =(z^\alpha-I_T)+(I_T-\mathcal Q_{T,h,\delta}).
\]
The first term is the left-tail error.  On the vertex scale below it is generally of order $\e^{-T}$, not smaller; after rescaling it becomes a $T$-independent part of $\mathcal H_{h,\delta}$.  Thus the profile is deliberately the profile of the complete target-minus-rule error, not merely of the quadrature error for the truncated integral.

For $X\le T$ and $|\theta|\le\Theta$ set
\begin{equation}\label{eq:vertex-scaling}
z_T(X,\theta)=\exp\!\left(\frac{X-T}{\alpha}+\ii\theta\right).
\end{equation}
Direct substitution into \eqref{eq:infinite-Q} gives the exact identity
\begin{equation}\label{eq:profile-identity}
 \e^T\bigl[z_T(X,\theta)^\alpha-
 \mathcal Q_{T,h,\delta}(z_T(X,\theta))\bigr]
=\mathcal H_{h,\delta}(X,\theta).
\end{equation}
Thus the scaled full error contains no hidden $T$-dependent remainder; in particular the left-tail term displayed above is already part of $\mathcal H_{h,\delta}$.  Since $X=T+\alpha\log|z|$, the converse change of variables also gives
\begin{equation}\label{eq:profile-sup}
 \e^T\norm{z^\alpha-\mathcal Q_{T,h,\delta}}_{\Sig}
=\sup_{X\le T,\ |\theta|\le\Theta} |\mathcal H_{h,\delta}(X,\theta)|.
\end{equation}

We record the two points needed to pass to the limit.  First, for $X$ in a compact interval the series in \eqref{eq:profile} converges locally uniformly: its tail is bounded by $C\e^{X/\alpha}\e^{-(1/\alpha-1)w_k}/w_k$. Second, if $0<\sigma\le\sigma_*$, then $q_0\ge1$ in Theorem~\ref{thm:quad}.  Combining that theorem for $m=0$ with the left-tail estimate following \eqref{eq:real-envelope} yields
\begin{equation}\label{eq:infinite-dense-bound}
 \norm{z^\alpha-\mathcal Q_{T,h,\delta}}_{\Sig}\le C\e^{-T}.
\end{equation}
Given any real $X$, choose $T\ge\max\{1,X\}$ in \eqref{eq:profile-identity}; \eqref{eq:infinite-dense-bound} then proves that $\mathcal H_{h,\delta}$ is bounded on all of $\mathbb R\times[-\Theta,\Theta]$.  For finite $T$, only the portion $X\le T$ of the profile is sampled; these domains increase to all of $\mathbb R$ as $T\to\infty$.  Hence the right-hand side of \eqref{eq:profile-sup} increases to the finite constant $C_{h,\delta,\Theta}$ in \eqref{eq:profile-constant}.

The constant is strictly positive.  As $X\to-\infty$, uniformly for $|\theta|\le\Theta$,
\begin{equation}\label{eq:profile-left}
 \mathcal H_{h,\delta}(X,\theta)
=\e^{X+\ii\alpha\theta}+O(\e^{X/\alpha}).
\end{equation}
Indeed, for sufficiently negative $X$ the denominator in each summand of \eqref{eq:profile} is comparable with $\e^{(w_k-X)/\alpha}$ and the remaining $k$-sum converges.  Since $1/\alpha>1$, the remainder is $o(\e^X)$ as $X\to-\infty$.  For any fixed $|\theta|\le\Theta$ we therefore have
\[
   \e^{-X}\mathcal H_{h,\delta}(X,\theta)\longrightarrow \e^{\ii\alpha\theta},
\]
so the profile is not identically zero and its supremum is strictly positive.  Hence
\begin{equation}\label{eq:infinite-exact}
 \norm{z^\alpha-\mathcal Q_{T,h,\delta}}_{\Sig}
=\bigl(C_{h,\delta,\Theta}+o(1)\bigr)\e^{-T}.
\end{equation}

It remains only to make the finite LP construction faithful at this scale. Put $\ell_T=\log(1+T)$ and choose
\begin{equation}\label{eq:profile-cutoffs}
S_T=\kappa(T+2\ell_T),\quad M_T=\left\lceil\frac{(T+S_T)^2}{h}+\delta-1\right\rceil,
 \quad
 \widetilde n_T=
 \left\lceil\frac{T+2\ell_T}{\log(2/R_\Theta)}\right\rceil.
\end{equation}
Truncate \eqref{eq:infinite-Q} at $M_T$, keep the first $N$ poles, and replace $N<k\le M_T$ by the degree-$\widetilde n_T$ Taylor polynomial about $z=1$ using \eqref{eq:exact-remainder}.  For the resulting $\widetilde r_{N,\delta}$, the estimate from Lemma~\ref{lem:bridge} gives
\begin{equation}\label{eq:profile-right-small}
 \norm{\mathcal Q_{T,h,\delta}-\widetilde r_{N,\delta}}_{\Sig}
 \le C{\e^{-T}}{(1+T)^{-2}}.
\end{equation}
The two pieces are deliberately matched to the same smaller scale.  For the discarded right tail, $w_{M_T+1}-T\ge S_T$ gives $\e^{-S_T/\kappa}=\e^{-T}(1+T)^{-2}$.  For the polynomial part,
\[
 \left({R_\Theta}/{2}\right)^{\widetilde n_T}
 \le \e^{-T}(1+T)^{-2}
\]
by the definition of $\widetilde n_T$, and~\eqref{eq:compression} then gives the same order.  Thus both finite-realization errors are $o(\e^{-T})$, as required.  Also $\widetilde n_T=O(T)=O(\sqrt N)$, while the first $N$ prescribed poles are unchanged.  Combining \eqref{eq:infinite-exact} and \eqref{eq:profile-right-small} gives the core-only sequence in Theorem~\ref{thm:exact}.

For the dimension-controlled sequence, $\sigma\le\sigma_*$ gives $\lambda_\sigma=1$.  Write $A_*:=A_{*,\lambda}$ from \eqref{eq:Astar}, retain through $K_{A_*}$, and use the zero-centered expansion \eqref{eq:zero-remainder} with
\begin{equation}\label{eq:profile-hybrid-degree}
 \widehat n_T=
 \left\lceil\frac{T+2\log(1+T)}{\log A_*}\right\rceil.
\end{equation}
Keep the same right cutoff $M_T$.  The same separation of scales is used here.  Estimate \eqref{eq:hybrid-weight-bound}, together with $A_*^{-\widehat n_T}\le \e^{-T}(1+T)^{-2}$, gives a compression error $O(\e^{-T}(1+T)^{-2})$, and the discarded right tail has the same order.  Hence replacing the infinite quadrature by the hybrid LP again changes the error only by $o(\e^{-T})$. Moreover,
\[
K_{A_*}-N=\sqrt{2\alpha N}+O(1),\qquad
 \widehat n_T=\sqrt{2\alpha N}+O(\log N),
\]
so \eqref{eq:exact-dimension} follows.  The resulting sequence $\widetilde r^{\rm hyb}_{N,\delta}$ therefore has the same constant in \eqref{eq:exact-asymptotic}.  Finally $T=\alpha\sigma\sqrt N$, completing the proof of Theorem~\ref{thm:exact}. \hfill$\square$

The theorem clarifies the dense/balanced branch: for this explicit, slightly overresolved sequence, $\alpha\sigma$ is the leading exponent with a positive constant, not merely an upper-bound exponent.  At balance, truncation and shifted-quadrature effects remain together in the profile \eqref{eq:profile}; we do not separate their phases by absolute values.

\section{Numerical tests and reproducibility}\label{sec:numerics}

The experiments test the exact profile normalization, the role of the shift, and allocation at fixed total dimension.  We keep numerical searches distinct from certified suprema and explicit coefficients distinct from fitted LPs.  For each construction, $E_N$ denotes the largest error returned by the documented boundary search on truncated sectors; the very-small-radius remainder is bounded separately.  Coefficients and function values use 70 decimal digits unless stated otherwise, whereas search coordinates are double precision.  The full grids, refinement rule, arithmetic, small-radius bound, and resolution study are given in Supplementary Section~SM1.  Repeating six representative cases at half and twice the production resolution changes the production-to-finest maxima by at most $\MeshMaxChange$.

\subsection{The exact asymptotic in the construction of Theorem~\ref{thm:exact}}
Take $\alpha=1/2$, $\Theta=\pi/3$, $\delta=1/2$, and compare $\sigma=3\sigma_*/4$ with balance $\sigma=\sigma_*$.  Table~\ref{tab:exact-num} compares the normalized boundary error $\e^T E_N$ with a separate numerical maximization of the stationary profile \eqref{eq:profile}.
\begin{table}[H]
 \centering\small
 \caption{Direct check of Theorem~\ref{thm:exact} for $\alpha=1/2$, $\Theta=\pi/3$, and midpoint shift. ``Profile search'' maximizes $|\mathcal H_{h,\delta}|$ numerically; the other columns are $\e^T E_N$ for the overresolved hybrid sequence.  Neither search is a certified continuous supremum.}
 \label{tab:exact-num}
 \begin{tabular}{@{}ccccc@{}}\toprule
 $\sigma/\sigma_*$ & Profile search & $N=36$ & $N=64$ & $N=100$\\
 \midrule
 $3/4$ & 1.54192734 & 1.54192734 & 1.54192734 & 1.54192734 \\
 $1$   & 2.43973532 & 2.43973532 & 2.43973532 & 2.43973532 \\
 \bottomrule\end{tabular}
\end{table}
The profile is searched on $X\in[-10,10]$, $0\le\theta\le\Theta$, then checked on $[-20,20]$ and refined on the selected angular side.  The relative window change is at most $\ProfileWindowChange$; an explicit series-tail bound controls truncation at selected points but does not certify the unsearched noncompact $X$-axis.  The normalized finite errors agree with the profile at the eight digits shown.  Supplementary Section~SM5 reports fixed-point finite-realization discrepancies, explains the distinct evaluation paths, and gives a nonmidpoint, nonsquare-root case $\alpha=2/5$, $\Theta=3\pi/5$, $\delta=3/4$.  These are consistency checks, not proofs of the global profile supremum.

\subsection{Interior shifts: a shared bound, different actual errors}
For $\alpha=1/2$, $\Theta=\pi/3$, and $\sigma=\sigma_*$, Figure~\ref{fig:shifts} shows the explicit construction at several shifts; $\delta=0$ is an endpoint reference outside Theorem~\ref{thm:main}.  At $N=100$, a scan in steps of $0.05$ over $0.05\le\delta\le0.95$ has its smallest sampled error at $\delta=\ScanBestDelta$, with ratio $\ScanBestRatio$ to the midpoint error $\ScanMidError$.  This is not a continuous optimization.  The kernel norm $A_\delta/A_{1/2}$ is symmetric and uniquely minimized at $1/2$, whereas the error need not be: truncation, quadrature, and compression carry phases absent from the absolute kernel bound.
\begin{figure}[H]
 \centering
 \includegraphics[width=.49\linewidth]{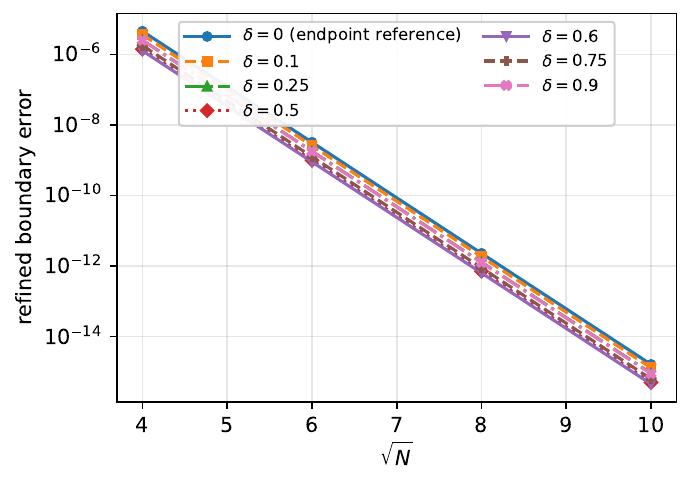}\hfill
 \includegraphics[width=.49\linewidth]{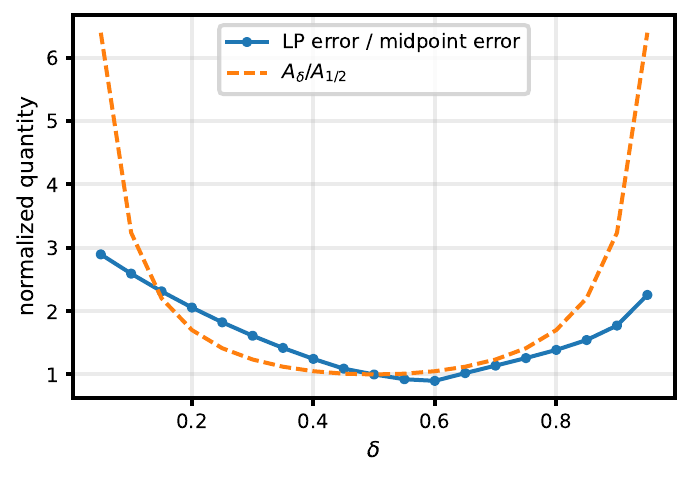}
 \caption{Shift effects. Left: constructive LP errors for fixed shifts at balanced spacing; the endpoint curve lies outside the theorem. Right: the $N=100$ scan, showing that actual error and diagonal-kernel norm need not have the same minimizing shift.}
 \label{fig:shifts}
\end{figure}

\subsection{Spacing, angle, and actual LP dimension}
For $\alpha=1/2$, $\Theta=\pi/3$, and midpoint shift, we fit $\log E_N$ against $\sqrt N$ at $N=36,64,100,144$ for six spacings.  The fitted rates follow $c_\Theta(\sigma)$ with largest relative discrepancy $\SpacingMaxDeviation$; the two-panel rate diagnostic is shown in Supplementary Section~SM1.  Repeating the same four $N$ values at balanced spacing while varying $\Theta$ gives Table~\ref{tab:angles}.  The sectors $\Theta=3\pi/5,3\pi/4,5\pi/6$ are genuinely reentrant because their openings $2\Theta$ exceed $\pi$, and the last two table columns show how costly direct Taylor compression becomes as $\Theta\uparrow\pi$. These finite-$N$ fits test the predicted branches rather than an exact asymptotic equality.  More importantly, $N$ alone hides the smooth-component cost, motivating the transition-pole split.

\begin{table}[!tbp]
 \centering\small
 \caption{Angle experiment at $\alpha=1/2$ and midpoint shift; $c_*=c_\Theta(\sigma_*)$. The last three columns use $N=100$; $n_N$ is the actual Taylor degree and $L=N+n_N+1$.}
 \label{tab:angles}
 {\setlength{\tabcolsep}{5.5pt}\begin{tabular}{@{}ccclrr@{}}
 \toprule
 $\Theta/\pi$ & $c_*$ & Fitted rate & $E_{100}$ & $n_{100}$ & $L_{100}$\\
 \midrule\AngleTableRows\bottomrule
 \end{tabular}}
\end{table}
The corresponding two-panel spacing/angle diagnostic is included in Supplementary Section~SM1.  We keep the angle table here because its last two columns expose the representation cost that the rate plot alone cannot show.

\subsection{Dimension control by a transition-pole band}\label{sec:dimension-num}
At balanced spacing, the direct-Taylor construction becomes expensive on wide sectors, while the transition split retains the same first $N$ poles and a much smaller stored representation.  For $N=100$, $\alpha=1/2$, and midpoint shift, the three reentrant cases $\Theta/\pi=3/5,3/4,5/6$ reduce the total dimensions $234,382,625$ to $121$; the largest relative change in the sampled error is only $\DimensionMaxErrorRatioDeviation$.  In the sparse test $\Theta=\pi/3$, $N=100$, and $\sigma=3\sigma_*/2$, rate matching changes the transition-split dimension from $\SparseUnitDimension$ to $\SparseMatchedDimension$ with error ratio $\SparseMatchedErrorRatio$.  These are finite-$N$ consistency checks; Supplementary Section~SM2 gives the complete counts and error-versus-dimension curves.

The preceding explicit-construction examples hold the core count $N$ fixed, so they isolate the cost of compressing the smooth tail.  A complementary question is whether the transition-pole allocation remains useful when the \emph{total} dimension is held fixed and both spaces are fitted.  Table~\ref{tab:matched-main} answers that question for four controlled pairs.  The core and hybrid spaces use identical training nodes, column scaling, relative SVD cutoff $10^{-13}$, and independent 70-digit boundary evaluation; both polynomial blocks use two-pass Arnoldi orthogonalization~\cite{BNT21}.  Here $E_{\rm core}$ and $E_{\rm hyb}$ denote the searched errors of the fitted functions.  The comparison therefore changes the allocation of the same dimension, not the numerical budget or search protocol.
\begin{table}[H]
 \centering\small
 \caption{Same-total-dimension fitted comparison at $\alpha=1/2$, $m=0$, $\delta=1/2$, and balanced spacing.  The ratio is $E_{\rm hyb}/E_{\rm core}$.  These four tests support the allocation mechanism of Proposition~\ref{prop:dimension}; they do not assert minimax optimality or uniform superiority.}
 \label{tab:matched-main}
 {\setlength{\tabcolsep}{5pt}\begin{tabular}{@{}ccclll@{}}\toprule
 $\Theta/\pi$ & $N$ & $L$ & $E_{\rm core}$ & $E_{\rm hyb}$ & Ratio\\\midrule
 \SpaceBenchmarkMainRows\bottomrule
 \end{tabular}}
\end{table}
Supplementary Section~SM4 gives the two allocations, retained numerical ranks, coefficient norms, singular-value diagnostics, and cutoff/evaluation checks.  In the four pairs above, the hybrid-to-core error ratio ranges from approximately $\SpaceBenchmarkMinRatio$ to $\SpaceBenchmarkMaxRatio$; this is empirical evidence for useful allocation at the tested parameters, not a general stability theorem.

\paragraph{Further diagnostics}
Supplementary Sections~SM3--SM4 collect the logarithmic-power experiments, the analytic-multiplier test, and the coefficient-level re-fitting diagnostics in a fixed hybrid space.  They support the scope of Theorem~\ref{thm:main} and separate the effect of re-fitting coefficients from the allocation comparison in Table~\ref{tab:matched-main}.

\section{Conclusions}\label{sec:conclusions}

The analysis answers three questions that are not settled by the known sector exponent.  First, interior shifts can be handled uniformly on compact subintervals of $(0,1)$: the one-sided Abel formula preserves the integrable endpoint singularity, while the local notch avoids an artificial algebraic loss near the moving pole line.  The same argument covers every fixed logarithmic power.  The midpoint minimizes the diagonal kernel norm, rather than the actual approximation error.

Second, for $z^\alpha$ in the dense and balanced regimes, vertex-scale rescaling turns the full infinite-rule error into a fixed profile and yields an explicit sequence with a positive leading constant; the later finite tail and compression are $o(\e^{-T})$.  Third, rate-matched transition-pole compression gives the singular-prototype dimension $N+2\sqrt{2\alpha\lambda_\sigma N}+O(1)$, with the smaller sparse-branch correction $\sigma_*/\sigma$.  The matched-dimension tests show that this allocation can matter substantially at finite $N$ on wide sectors.

These conclusions are deliberately construction-specific.  They do not provide sector minimax lower bounds, finite-degree shift optimality, exact constants for general analytic multipliers and logarithmic powers, or a floating-point stability theorem.  Sparse-branch exact asymptotics remain open.\smallskip

\textit{Declarations.} OpenAI ChatGPT (GPT-5.6 Sol and GPT-6 Astra) was used to suggest candidate derivation routes and assist with consistency checks for parts of the shifted-quadrature, notched-contour, exact-asymptotic, and dimension-control arguments, as well as literature comparison, exposition, and developing and reviewing numerical verification code. The authors assume responsibility for all content.

\clearpage
\setcounter{page}{1}
\setcounter{section}{0}
\setcounter{subsection}{0}
\setcounter{subsubsection}{0}
\setcounter{equation}{0}
\setcounter{figure}{0}
\setcounter{table}{0}
\setcounter{theorem}{0}
\renewcommand{\thepage}{SM\arabic{page}}
\renewcommand{\thesection}{SM\arabic{section}}
\renewcommand{\thesubsection}{\thesection.\arabic{subsection}}
\renewcommand{\thesubsubsection}{\thesubsection.\arabic{subsubsection}}
\renewcommand{\theequation}{\thesection.\arabic{equation}}
\renewcommand{\thefigure}{\thesection.\arabic{figure}}
\renewcommand{\thetable}{\thesection.\arabic{table}}
\renewcommand{\thetheorem}{\thesection.\arabic{theorem}}
\begin{center}
  {\bfseries\Large Supplementary Materials:\par}
  \medskip
  {\bfseries\large
  Shifted tapered-pole approximation on finite sectors:
  an exact asymptotic and dimension control\par}
  \medskip
  {Fei Xue and Tianqi Zhang}
\end{center}
\bigskip
\input{supplement_body.tex}

\end{document}

%% file: numerical_macros.tex
\newcommand{\ScanBestDelta}{0.60}
\newcommand{\ScanBestRatio}{0.8993}
\newcommand{\ScanMidError}{5.0239\times10^{-16}}
\newcommand{\SpacingMaxDeviation}{7.08\times10^{-5}}
\newcommand{\AngleTableRows}{$0$ & 4.44288 & 4.44288 & $2.29\times10^{-19}$ & 65 & 166 \\
$1/12$ & 4.25374 & 4.25374 & $1.13\times10^{-18}$ & 62 & 163 \\
$1/3$ & 3.62760 & 3.62759 & $5.02\times10^{-16}$ & 53 & 154 \\
$1/2$ & 3.14159 & 3.14158 & $6.52\times10^{-14}$ & 91 & 192 \\
$3/5$ & 2.80993 & 2.80947 & $1.75\times10^{-12}$ & 133 & 234 \\
$3/4$ & 2.22144 & 2.22144 & $6.67\times10^{-10}$ & 281 & 382 \\
$5/6$ & 1.81380 & 1.81380 & $4.21\times10^{-8}$ & 524 & 625 \\}
\newcommand{\MeshTableRows}{1 & $6.32382\times10^{-16}$ & $6.32382\times10^{-16}$ & $6.32382\times10^{-16}$ & $2.1\times10^{-17}$ \\
2 & $2.29317\times10^{-19}$ & $2.29317\times10^{-19}$ & $2.29317\times10^{-19}$ & $4.6\times10^{-17}$ \\
3 & $1.74946\times10^{-12}$ & $1.74946\times10^{-12}$ & $1.74946\times10^{-12}$ & $0$ \\
4 & $6.66906\times10^{-10}$ & $6.66906\times10^{-10}$ & $6.66906\times10^{-10}$ & $4.1\times10^{-14}$ \\
5 & $4.20813\times10^{-8}$ & $4.20813\times10^{-8}$ & $4.20813\times10^{-8}$ & $2.6\times10^{-14}$ \\
6 & $7.27657\times10^{-14}$ & $7.27657\times10^{-14}$ & $7.27657\times10^{-14}$ & $3.0\times10^{-14}$ \\}
\newcommand{\MeshMaxChange}{4.06\times10^{-14}}

\newcommand{\MultiplierTableRows}{36 & 72 & 17 & 126 & $1.56\times10^{-3}$ \\
64 & 95 & 19 & 179 & $1.75\times10^{-5}$ \\
100 & 119 & 21 & 241 & $1.77\times10^{-7}$ \\}

%% file: dimension_macros.tex
\newcommand{\DimensionMaxErrorRatioDeviation}{\ensuremath{1.09\times10^{-4}}}

\newcommand{\SparseUnitDimension}{121}
\newcommand{\SparseMatchedDimension}{115}

\newcommand{\SparseMatchedErrorRatio}{0.999942}

%% file: exact_asymptotic_macros.tex
\newcommand{\ProfileWindowChange}{\ensuremath{2.23\times10^{-18}}}
\newcommand{\ProfileSummaryRows}{%
A & $3/4$ & 1.5419273423 & 1.246732 & $2.23\times10^{-18}$ & $3.24\times10^{-43}$ \\
A & $1$ & 2.4397353188 & 1.750240 & $2.02\times10^{-18}$ & $8.70\times10^{-43}$ \\
B & $3/4$ & 1.3198583855 & 0.708347 & $6.12\times10^{-17}$ & $3.54\times10^{-65}$ \\
B & $1$ & 2.6802864919 & 1.366884 & $8.20\times10^{-18}$ & $1.87\times10^{-64}$ \\
}
\newcommand{\ProfileFiniteRows}{%
A & $3/4$ & 36 & 42 & 9 & 52 & 1.5419273423 & 1.246732 \\
A & $3/4$ & 64 & 72 & 11 & 84 & 1.5419273423 & 1.246732 \\
A & $3/4$ & 100 & 110 & 13 & 124 & 1.5419273423 & 1.246732 \\
A & $1$ & 36 & 42 & 8 & 51 & 2.4397353188 & 1.750240 \\
A & $1$ & 64 & 72 & 10 & 83 & 2.4397353188 & 1.750240 \\
A & $1$ & 100 & 110 & 12 & 123 & 2.4397353188 & 1.750240 \\
B & $3/4$ & 36 & 42 & 8 & 51 & 1.3198583855 & 0.708347 \\
B & $3/4$ & 64 & 72 & 10 & 83 & 1.3198583855 & 0.708347 \\
B & $3/4$ & 100 & 109 & 12 & 122 & 1.3198583855 & 0.708347 \\
B & $1$ & 36 & 42 & 8 & 51 & 2.6802864919 & 1.366884 \\
B & $1$ & 64 & 72 & 10 & 83 & 2.6802864919 & 1.366884 \\
B & $1$ & 100 & 109 & 12 & 122 & 2.6802864919 & 1.366884 \\
}

%% file: space_benchmark_macros.tex
\newcommand{\SpaceBenchmarkMinRatio}{0.0009}
\newcommand{\SpaceBenchmarkMaxRatio}{0.0071}
\newcommand{\SpaceBenchmarkMeshChange}{\ensuremath{4.09\times10^{-14}}}

\newcommand{\SpaceBenchmarkMainRows}{%
$3/4$ & 64 & 81 & $9.233\times10^{-6}$ & $5.258\times10^{-8}$ & 0.0057 \\
$3/4$ & 100 & 121 & $6.947\times10^{-7}$ & $6.187\times10^{-10}$ & 0.0009 \\
$5/6$ & 64 & 81 & $2.054\times10^{-4}$ & $1.463\times10^{-6}$ & 0.0071 \\
$5/6$ & 100 & 121 & $3.123\times10^{-5}$ & $3.889\times10^{-8}$ & 0.0012 \\
}

%% file: supplement_macros.tex
\newcommand{\HybridDiagnosticsRows}{%
$1/3$ & $7.098\times10^{-13}$ & $5.764\times10^{-13}$ & 76 & $8.127\times10^{-1}$ & 1357 \\
$3/5$ & $4.824\times10^{-10}$ & $4.348\times10^{-10}$ & 78 & $2.228\times10^{0}$ & 1270 \\
$3/4$ & $5.670\times10^{-8}$ & $5.257\times10^{-8}$ & 80 & $2.596\times10^{3}$ & 1207 \\
$5/6$ & $1.583\times10^{-6}$ & $1.463\times10^{-6}$ & 81 & $1.303\times10^{4}$ & 1164 \\
}
\newcommand{\CompactBenchmarkRows}{%
$3/4$ & 64 & 81 & $9.233\times10^{-6}$ & $5.258\times10^{-8}$ & $5.695\times10^{-3}$ & (81,80) \\
$3/4$ & 100 & 121 & $6.947\times10^{-7}$ & $6.187\times10^{-10}$ & $8.905\times10^{-4}$ & (121,117) \\
$5/6$ & 64 & 81 & $2.054\times10^{-4}$ & $1.463\times10^{-6}$ & $7.124\times10^{-3}$ & (81,81) \\
$5/6$ & 100 & 121 & $3.123\times10^{-5}$ & $3.889\times10^{-8}$ & $1.245\times10^{-3}$ & (121,119) \\
}
\newcommand{\BenchmarkDiagnosticRows}{%
$3/4$ & 64 & core & 1207 & $2.994\times10^{1}$ & $3.95\times10^{-5}$ \\
$3/4$ & 64 & hybrid & 1207 & $4.536\times10^{4}$ & $1.43\times10^{-13}$ \\
$3/4$ & 100 & core & 1587 & $3.542\times10^{1}$ & $2.70\times10^{-6}$ \\
$3/4$ & 100 & hybrid & 1587 & $1.909\times10^{1}$ & $4.36\times10^{-12}$ \\
$5/6$ & 64 & core & 1164 & $2.874\times10^{1}$ & $4.30\times10^{-4}$ \\
$5/6$ & 64 & hybrid & 1164 & $1.900\times10^{5}$ & $3.30\times10^{-13}$ \\
$5/6$ & 100 & core & 1532 & $3.407\times10^{1}$ & $5.54\times10^{-5}$ \\
$5/6$ & 100 & hybrid & 1532 & $7.674\times10^{3}$ & $3.14\times10^{-13}$ \\
}

\newcommand{\BenchmarkEvaluationCheck}{%
$1.45\times10^{-61}$
}

%% file: supplement_body.tex
This supplementary part gives the numerical specifications, additional experiments, and data organization for the main article. All mathematical statements of the article, including their proofs, remain in the main text. The reported errors and profile diagnostics are numerical evaluations or searches, not interval-arithmetic certifications. References to equations and theorems in the main article are given by their displayed numbers; this supplement's own equations, tables, and figures are numbered with the SM prefix.

For readability, we repeat the basic numerical notation used throughout the supplement.  The singularity exponent satisfies $0<\alpha<1$, $m$ is the fixed nonnegative logarithmic power, $0\le\Theta<\pi$ is the sector half-angle, $N$ is the number of prescribed core poles, $\sigma>0$ is their taper parameter, and $0<\delta<1$ is the interior grid shift.  Write
\[
 \begin{aligned}
 \sigma_*&=2\sqrt{\frac{\pi(\pi-\Theta)}{\alpha}},
 \qquad h=\alpha^2\sigma^2,
 \qquad T=\sqrt{Nh}=\alpha\sigma\sqrt N,\\
 \lambda_\sigma&=\min\{1,(\sigma_*/\sigma)^2\},
 \qquad c_\Theta(\sigma)=\min\left\{\alpha\sigma,
 \frac{4\pi(\pi-\Theta)}{\sigma}\right\}.
 \end{aligned}
\]
Thus $\sigma=\sigma_*$ is balanced spacing and then $\lambda_\sigma=1$.  We use $E_N$ for the searched boundary error of the construction under discussion.  Whenever a representation has $K$ retained rational functions and a polynomial of degree $n$, its counted LP dimension is $L=K+n+1$.

\section{Boundary-search specification and resolution}\label{sec:search}
\subsection{Arithmetic, construction, and search grids}
The recorded environment for the reported experiments is Python 3.13.5 with mpmath 1.3.0, NumPy 2.3.5, SciPy 1.17.0, and Matplotlib 3.10.8. Rational parameter strings are converted directly to multiprecision numbers. Explicit LP coefficients and objective values use 70 decimal digits; the overresolved exact-asymptotic tests in Section~\ref{sec:profiles} use 80. The least-squares solves use binary64 arithmetic, followed by 70-digit evaluation of their fixed coefficients. Search coordinates and the scalar optimizer use binary64 arithmetic. Increasing the evaluation precision does not change the precision of a preceding least-squares solve.

For the direct construction, $h$ and $T$ are as above, $M$ is the right-tail cutoff, $b_k$ is the $k$th quadrature weight, and $n_N$ is the polynomial degree; their precise formulas are (3.1), (3.2), and (3.7) in the main article. The ordinary hybrid construction uses (3.10) at the optimized threshold (3.12). In particular, with $\lambda=\lambda_\sigma$,
\begin{equation}\label{eq:counts}
 K=\left\lfloor N+\sqrt{2\alpha\lambda N}+\frac{\alpha\lambda}{2}+\delta\right\rfloor,
 \qquad n=\left\lceil\sqrt{2\alpha\lambda N}\right\rceil.
\end{equation}
Exact rational arithmetic is used for these floor and ceiling operations when the square root is rational, avoiding off-by-one errors at integer boundaries. For the overresolved constructions, (3.26) gives the right cutoff and the core-only polynomial degree; (3.28) gives the hybrid polynomial degree.

On the upper radial side write $z=r\e^{\ii\Theta}$, set $t=-\log r$, and hence $X=T-\alpha t$. For each case the radial search ends at
\begin{equation}\label{eq:tmax}
 t_{\max}=\sigma\sqrt N+
 \max\left\{32,\frac{28+m\log(2+\sigma\sqrt N)}{\alpha}\right\}.
\end{equation}
The radial grid is the union of a uniform $t$ grid, the endpoint $t_{\max}$, the 200 linear-radius points $r=j/200$ ($1\le j\le200$), and the phase-resolving points
\[
 t=\frac{T-\sqrt{vh}}{\alpha},\qquad
 v=\frac{j}{s},\quad 0\le j\le sN.
\]
Here $v=X^2/h$ is dimensionless, so this additional grid covers $X\ge0$; $s$ is the phase subdivision count in Table~\ref{tab:search-levels}. The upper arc is sampled uniformly in angle, and the vertex is evaluated separately. Conjugation symmetry supplies the lower boundary.
\begin{table}[!htbp]
 \centering\small
 \caption{Three search levels. The arc count is $J=\max\{J_0,\lceil q\max(\Theta,0.1)(d+1)/\pi\rceil+1\}$, where $J_0$ is the minimum arc count, $q$ is the degree-dependent oversampling factor, and $d$ is the prototype polynomial degree (before multiplication by the analytic-multiplier polynomial introduced in Section~\ref{sec:multiplier}). For the same-dimension comparison, both spaces use the same larger degree $d=L-N-1$.}
 \label{tab:search-levels}
 \begin{tabular}{@{}ccccc@{}}\toprule
 Level & Logarithmic step & $s$ & $J_0$ & $q$\\\midrule
 Coarse & $1/2$ & 4 & 81 & 2\\
 Production & $1/4$ & 8 & 161 & 4\\
 Fine & $1/8$ & 16 & 321 & 8\\\bottomrule
 \end{tabular}
\end{table}
\FloatBarrier
Every sampled local maximum exceeding $0.6$ times the largest sampled error, including the vertex in this screening value, is refined within its two neighboring sample locations. The implementation uses SciPy's bounded scalar minimizer, with \texttt{xatol}$=10^{-10}$ and at most 80 iterations, on the negative error divided by the sampled peak value. Endpoint values remain candidates. This procedure is a numerical search even when successive resolutions agree to many digits.

\subsection{A bound for the omitted small-radius disk}\label{sec:disk-bound}
The unsampled disk is not simply discarded. Put $\varepsilon=\e^{-t_{\max}}$; for all recorded cases it is smaller than every retained positive pole modulus. Since $t_{\max}>m/\alpha$,
\[
 \sup_{|z|\le\varepsilon,\ |\arg z|\le\Theta}
 |z^\alpha(\log z)^m|
 \le F_\varepsilon:=\e^{-\alpha t_{\max}}(t_{\max}+\Theta)^m.
\]
To state the bound independently of the particular representation, write the computed approximation as
\[
 R(z)=\sum_k\eta_k\phi_k(z)+p(z),
\]
where each $\phi_k$ is either $z/(z+a_k)$ or $a_k/(z+a_k)$ and all retained $a_k$ are positive.  For either rational basis function, the variation from its value at zero is at most $\varepsilon/(a_k-\varepsilon)$. Thus the rational contribution is bounded by
\[
 V_{\rm rat}=\sum_k |\eta_k|\frac{\varepsilon}{a_k-\varepsilon}.
\]
For $p(z)=\sum_j c_jz^j$, set $V_{\rm pol}=\sum_{j\ge1}|c_j|\varepsilon^j$. For the direct Taylor representation $p(z)=\sum_j c_j(z-1)^j$, use instead
\[
 V_{\rm pol}=\sum_{j\ge1}j|c_j|(1+\varepsilon)^{j-1}\varepsilon.
\]
Consequently the unsearched error is bounded by
\begin{equation}\label{eq:disk-bound}
 F_\varepsilon+|R(0)|+V_{\rm rat}+V_{\rm pol}.
\end{equation}
For the multiplier experiment, write $R=q_N R^0$, where $R^0$ denotes the already constructed singular-prototype approximation and $q_N$ is the Taylor polynomial of $\e^z$ specified in Section~\ref{sec:multiplier}. Here $V_{\rm rat}$ and $V_{\rm pol}$ are computed for $R^0$. Since $|q_N(z)|\le\e^\varepsilon$ and $|q_N(z)-1|\le\varepsilon\e^\varepsilon$, the corresponding bound is
\[
 \e^\varepsilon F_\varepsilon+|R^0(0)|+
 \e^\varepsilon(V_{\rm rat}+V_{\rm pol})+
 |R^0(0)|\varepsilon\e^\varepsilon.
\]
These inequalities control only the omitted disk. The numerical values of the bounds are evaluated in multiprecision arithmetic, without outward rounding; neither they nor the searches on the remaining boundary constitute an interval-certified supremum.

\subsection{Resolution check with all parameters specified}
The six resolution cases are listed in Table~\ref{tab:mesh-cases}; there are no implicit defaults. All use direct Taylor compression, $g=1$, and balanced spacing. Table~\ref{tab:mesh} reports their searches. The largest relative production-to-fine change is $\MeshMaxChange$.
\begin{table}[!htbp]
 \centering\small
 \caption{Parameters for the resolution study.}\label{tab:mesh-cases}
 \begin{tabular}{@{}ccccccc@{}}\toprule
 Case & $\alpha$ & $m$ & $\Theta/\pi$ & $\delta$ & $N$ & $\sigma/\sigma_*$\\\midrule
 1 & $1/2$ & 0 & $1/3$ & $3/4$ & 100 & 1\\
 2 & $1/2$ & 0 & 0 & $1/2$ & 100 & 1\\
 3 & $1/2$ & 0 & $3/5$ & $1/2$ & 100 & 1\\
 4 & $1/2$ & 0 & $3/4$ & $1/2$ & 100 & 1\\
 5 & $1/2$ & 0 & $5/6$ & $1/2$ & 100 & 1\\
 6 & $2/5$ & 3 & $1/3$ & $1/2$ & 196 & 1\\\bottomrule
 \end{tabular}
\end{table}
\FloatBarrier
\begin{table}[!htbp]
 \centering\small
 \caption{Refined boundary maxima at three resolutions. The last column is $|E^{(1)}/E^{(2)}-1|$.}\label{tab:mesh}
 \begin{tabular}{@{}cllll@{}}\toprule
 Case & Coarse $E^{(0)}$ & Production $E^{(1)}$ & Fine $E^{(2)}$ & Change\\
 \midrule\MeshTableRows\bottomrule
 \end{tabular}
\end{table}
\FloatBarrier

\subsection{Spacing and angle diagnostics}\label{sec:rate-diagnostics}
The main article keeps the tabulated angle experiment because it also records the polynomial degree and total dimension.  For completeness, Figure~\ref{fig:rates-supp} collects the corresponding rate-only diagnostics.  These fits verify the expected branch behavior over the tested finite values of $N$; they are not used to prove the convergence theorem or the exact profile constant.
\begin{figure}[!htbp]
 \centering
 \includegraphics[width=.47\linewidth]{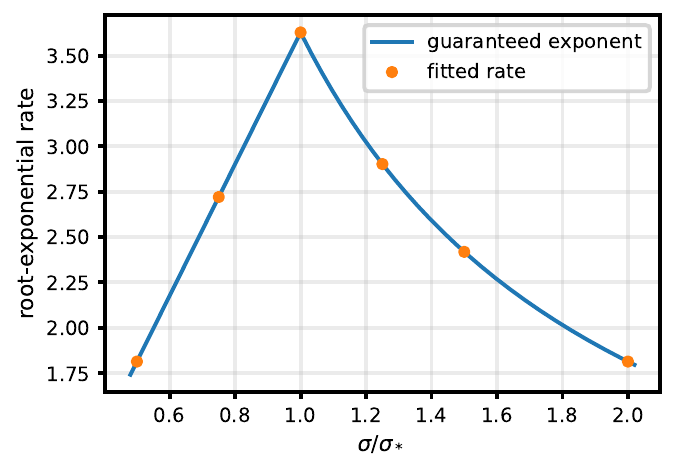}\hfill
 \includegraphics[width=.47\linewidth]{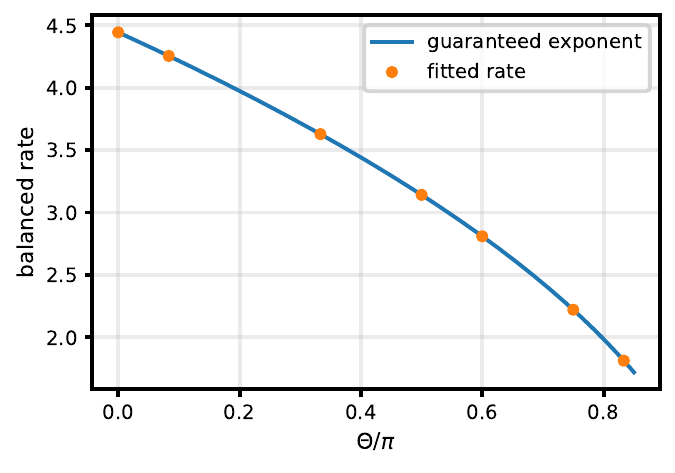}
 \caption{Root-exponential rate diagnostics. Left: fitted rates versus the two-branch guarantee for $\alpha=1/2$, $\Theta=\pi/3$. Right: the balanced angular law, including three reentrant sectors.}
 \label{fig:rates-supp}
\end{figure}
\FloatBarrier

\section{Dimension-control curves}\label{sec:curves}
For $\alpha=1/2$, $\delta=1/2$, balanced spacing, and the reentrant sectors $\Theta=3\pi/5,3\pi/4,5\pi/6$, Figure~\ref{fig:dimension} compares direct Taylor compression with the transition-pole construction at $N=36,64,100,144$. The paired methods have the same first $N$ tapered poles and nearly equal searched errors, but different representation dimensions. This is a comparison of two explicit realizations, not of best fits; the latter question is addressed separately in Section~\ref{sec:space-comparison}.
\begin{figure}[!htbp]
 \centering\includegraphics[width=.86\linewidth]{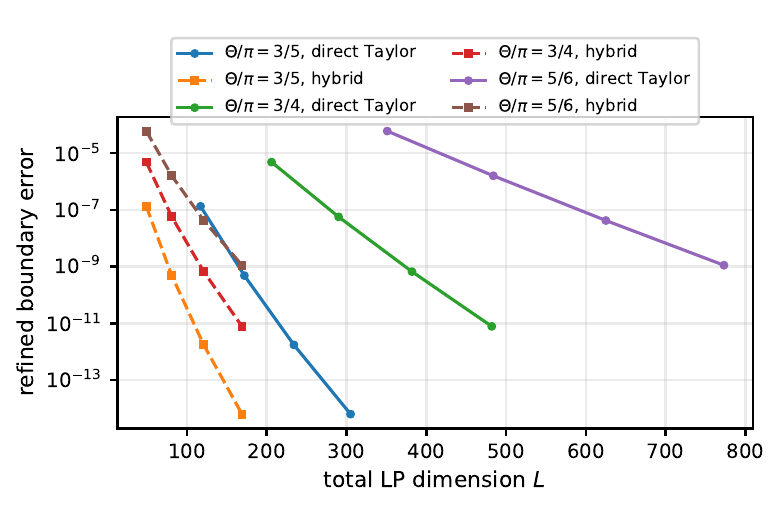}
 \caption{Wide-sector errors versus total LP dimension for direct Taylor compression and the optimized hybrid split. Parameters are $\alpha=1/2$, $\delta=1/2$, $\sigma=\sigma_*$, and $N=36,64,100,144$. The legend is outside the data region.}\label{fig:dimension}
\end{figure}
\FloatBarrier

\section{Fixed logarithmic powers and an analytic multiplier}\label{sec:logs}
For $\alpha=2/5$, $\Theta=\pi/3$, $\delta=1/2$, and $m=0,1,2,3$, Figure~\ref{fig:logpowers} shows all four core counts $N=64,100,144,196$. The main article reports the endpoint values at $N=196$. The balanced normalization includes $(1+T)^m$; the sparse normalization deliberately does not, in accordance with the branch-specific upper bounds.
\begin{figure}[!htbp]
 \centering
 \includegraphics[width=.47\linewidth]{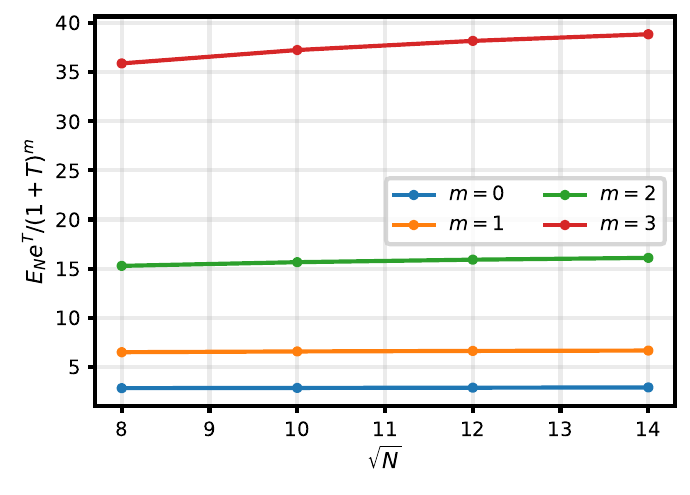}\hfill
 \includegraphics[width=.47\linewidth]{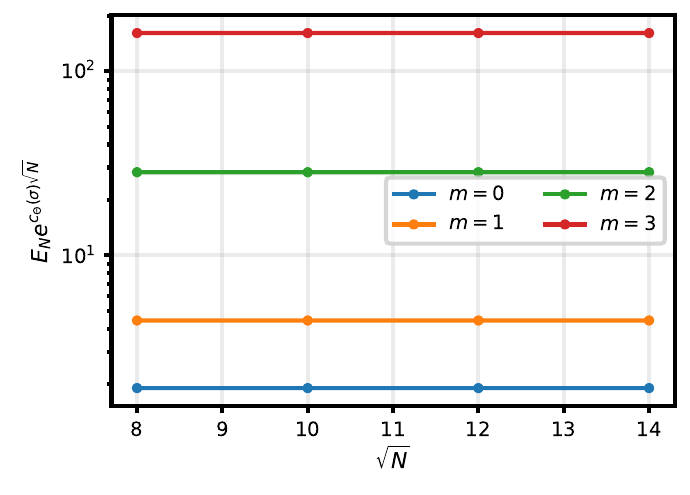}
 \caption{Logarithmic-power diagnostics for $\alpha=2/5$, $\Theta=\pi/3$, $\delta=1/2$, and $m=0,1,2,3$. Left: $E_N\e^T/(1+T)^m$ at balance. Right: $E_N\e^{c_\Theta(\sigma)\sqrt N}$ at $\sigma=3\sigma_*/2$. Each curve uses $N=64,100,144,196$.}\label{fig:logpowers}
\end{figure}
\FloatBarrier

\subsection{An exponential multiplier}\label{sec:multiplier}
Here $\alpha=2/5$, $m=2$, $\Theta=3\pi/5$, $\delta=1/2$, and $\sigma=\sigma_*$. We first form the direct singular-prototype LP $r^0_{N,\delta}$ in (3.16), and then multiply it by
\[
 q_N(z)=\sum_{j=0}^{d_N}\frac{z^j}{j!},\qquad
 d_N=\min\left\{d\ge0:\frac{\e}{(d+1)!}\le\e^{-T-20}\right\}.
\]
The displayed bound follows from the exponential-series remainder on $|z|\le1$. It makes multiplier approximation smaller than the target LP scale. The error searched in Table~\ref{tab:multiplier} is that of $q_Nr^0_{N,\delta}$ against $\e^zz^{2/5}(\log z)^2$, not that of an independently fitted product. By (3.18), the product has the same $N$ core poles and polynomial degree at most $n_N+d_N$. We count the prescribed representation as $L=N+n_N+d_N+1$.
\begin{table}[!htbp]
 \centering\small
 \caption{Analytic-multiplier experiment. The columns separate the prototype's direct Taylor degree $n_N$ from the multiplier degree $d_N$.}\label{tab:multiplier}
 \begin{tabular}{@{}crrrl@{}}\toprule
 $N$ & $n_N$ & $d_N$ & $L$ & Searched error\\\midrule
 \MultiplierTableRows\bottomrule
 \end{tabular}
\end{table}
\FloatBarrier

\section{Least-squares diagnostics and matched-space comparison}\label{sec:fitting}
\subsection{Training nodes and coefficient conventions}
For a space with $K$ rational functions, polynomial degree $n$, and total dimension $L=K+n+1$, let $P_{\rm tr}$ denote the total number of complex training nodes. The training set consists of upper-ray logarithmic samples with step $0.15$ up to $\sigma\sqrt N+24$ (including the first grid point past this value), 161 equally spaced upper-ray radii including zero and one, and $\max\{241,8L\}$ uniformly spaced upper-arc angles including both endpoints. Repeated nodes are retained, so these choices also specify the relative sampling weights. Both spaces in every matched pair use exactly the same nodes, in the same order.

The rational columns are $a_k/(z+a_k)$, where $-a_k<0$ is the $k$th retained pole and $a_k>0$ is its modulus. For the legacy fitted-space diagnostic they are followed by $1,z,\ldots,z^n$. Real and imaginary equations are stacked, and each column of the resulting real matrix is scaled to sampled Euclidean norm one. SciPy's \texttt{gelsd} least-squares driver uses a relative singular-value cutoff $10^{-13}$. Rational pole moduli are rounded to binary64 for matrix assembly; the computed coefficients are then held fixed and evaluated against the prescribed 70-digit pole moduli. This convention is the same for both spaces. The accompanying verification records store both sets of pole moduli and also check evaluation with the frozen assembly poles. The small-radius bound in Section~\ref{sec:disk-bound} also applies to these fixed fitted functions.

A coefficient norm below always means the Euclidean norm \emph{after undoing column normalization}, in the explicitly stated basis. The authors retain the column scales, or the corresponding scaled-coordinate norm where reported, as part of the underlying computational records. Norms in a monomial basis must not be compared directly with those in the Arnoldi basis introduced below.

\subsection{Re-fitting in the same hybrid space}
This reproduces the original cross-check, with $\alpha=1/2$, $m=0$, $\delta=1/2$, $N=64$, $\sigma=\sigma_*$, and $(K,n,L)=(72,8,81)$. Table~\ref{tab:hybrid-fit} shows modest improvements over the explicit coefficients. Its norms use the basis $\{a_k/(z+a_k)\}_{k=1}^{72}\cup\{z^j\}_{j=0}^{8}$, before column normalization. The large norms at the widest angles are therefore basis-dependent diagnostics, not invariant condition numbers or a stability theorem.
\begin{table}[!htbp]
 \centering\small
 \caption{Constructive and fitted errors in the same hybrid space. $P_{\rm tr}$ is the number of complex training nodes; the real system has $2P_{\rm tr}$ rows.}\label{tab:hybrid-fit}
 \begin{tabular}{@{}cllrlr@{}}\toprule
 $\Theta/\pi$ & Constructive error & Fitted error & Rank & $\norm{c}_2$ & $P_{\rm tr}$\\\midrule
 \HybridDiagnosticsRows\bottomrule
 \end{tabular}
\end{table}
\FloatBarrier

\subsection{A same-dimension comparison of two allocations}\label{sec:space-comparison}
The preceding experiment changes coefficients but not the space. We now compare two spaces at the same total dimension. In both, $\alpha=1/2$, $m=0$, $\delta=1/2$, and $\sigma=\sigma_*$. The conventional space keeps the first $N$ tapered poles and uses degree $L-N-1$. The hybrid space keeps $K>N$ poles and uses the smaller degree $n$ in \eqref{eq:counts}. For $N=64$ the two allocations are $(K,n)=(64,16)$ and $(72,8)$, both with $L=81$; for $N=100$ they are $(100,20)$ and $(110,10)$, both with $L=121$.

To avoid handicapping the conventional space through its larger monomial block, both polynomial blocks are generated by two-pass Arnoldi orthogonalization, following the principle of Vandermonde with Arnoldi~\cite{BNT21s}. On the complex upper-boundary samples, we use the real inner product
\[
 \langle u,v\rangle_{\mathbb R}=\operatorname{Re}\sum_{j=1}^{P_{\rm tr}}\overline{u_j}v_j.
\]
Starting with $\varphi_0=1/\sqrt{P_{\rm tr}}$, apply multiplication by the sample coordinates and two passes of modified Gram--Schmidt to generate $\varphi_1,\varphi_2,\ldots$. This produces real recurrence coefficients, hence real-coefficient polynomials. After the same column normalization and truncated SVD as above, the frozen recurrence is converted to monomial coefficients at 70 digits solely for evaluation. No additional fitting is performed in that conversion. The authors retain the recurrence, native fitted coefficients, singular values, and all training nodes as part of the underlying computational records.

Both members of a pair use identical production and fine search grids, including the common larger polynomial degree $d=L-N-1$ when choosing the arc grid. The largest production-to-fine relative change is $\SpaceBenchmarkMeshChange$. Re-evaluating the same fixed functions at 90 rather than 70 digits changes the error at the selected maximum by at most \BenchmarkEvaluationCheck\ relatively. These checks concern sampling and evaluation, not the accuracy of the original binary64 solves.
\begin{table}[!htbp]
 \centering\small
 \caption{Matched-space comparison. $E_{\rm core}$ and $E_{\rm hyb}$ use fitted coefficients in both spaces. The last column gives the two retained numerical ranks, in the same order.}\label{tab:benchmark}
 {\setlength{\tabcolsep}{3.5pt}\begin{tabular}{@{}ccclllc@{}}\toprule
 $\Theta/\pi$ & $N$ & $L$ & $E_{\rm core}$ & $E_{\rm hyb}$ & Ratio & Ranks\\\midrule
 \CompactBenchmarkRows\bottomrule
 \end{tabular}}
\end{table}
\FloatBarrier
\begin{table}[!htbp]
 \centering\small
 \caption{Diagnostics for Table~\ref{tab:benchmark}. The coefficient norm is in the native rational--Arnoldi basis, after undoing column scaling. The singular-value ratio uses the smallest \emph{retained} value and the largest value of the column-normalized real matrix.}\label{tab:benchmark-diagnostics}
 \begin{tabular}{@{}cccrll@{}}\toprule
 $\Theta/\pi$ & $N$ & Space & $P_{\rm tr}$ & $\norm{c}_2$ & $s_{\min,\rm ret}/s_{\max}$\\\midrule
 \BenchmarkDiagnosticRows\bottomrule
 \end{tabular}
\end{table}
\FloatBarrier
The core systems retain all $L$ singular directions; their error is not a consequence of truncating an ill-conditioned monomial block. The hybrid spaces have smaller searched errors in these four matched tests, despite some redundant directions being removed. This provides an allocation comparison beyond the conservative direct-Taylor construction. It does not establish a minimax dimension law, uniform superiority at other parameters, or a floating-point stability result.

\section{Stationary profiles beyond the midpoint square root}\label{sec:profiles}
We use two parameter families:
\[
 \mathrm A:\ (\alpha,\Theta,\delta)=(1/2,\pi/3,1/2),\qquad
 \mathrm B:\ (\alpha,\Theta,\delta)=(2/5,3\pi/5,3/4).
\]
Each is tested at $\sigma/\sigma_*=3/4$ and $1$, with $N=36,64,100$. Family A reproduces the main table; family B combines a nonsquare-root exponent, a nonmidpoint shift, and a reentrant sector. The finite functions use the overresolved hybrid realization of Theorem~1.2, including both $2\log(1+T)$ margins. Their errors are evaluated at 80 digits, independently of the profile searches.

\subsection{Profile search and the scope of its tail bound}
For each family and spacing, the two-dimensional profile is searched on $[-10,10]\times[0,\Theta]$ and on $[-20,20]\times[0,\Theta]$. The exploratory search uses SciPy's differential evolution with seed 7, population multiplier 10, at most 300 iterations, relative tolerance $10^{-9}$, and polishing. Each angular side is also sampled at 321 uniformly spaced $X$ values. Side-local peaks reaching $0.8$ of that side's sampled maximum are refined with bounded scalar minimization (\texttt{xatol}$=10^{-12}$, at most 100 iterations). Window endpoints and the two-dimensional candidate are retained. The six largest candidate values from this exploratory stage are re-evaluated at 80 digits. All selected maxima in these four experiments lie on $\theta=\Theta$; this is an observed feature, not an assumption imposed on the two-dimensional search.

Set $c_\alpha=\sin(\pi\alpha)/(\pi\alpha)$. The profile series is truncated at $J_{\rm prof}=\lceil 100^2/h+\delta\rceil$. Let $\mathcal H^{(J_{\rm prof})}_{h,\delta}$ denote the function obtained from the profile (1.7) by retaining only the terms $1\le k\le J_{\rm prof}$ in its $k$-series. Put $\kappa=\alpha/(1-\alpha)$ and $w_{J_{\rm prof}+1}=\sqrt{(J_{\rm prof}+1-\delta)h}$. Sector separation and a decreasing sum--integral comparison give the explicit bound
\begin{equation}\label{eq:profile-tail}
 \left|\mathcal H_{h,\delta}(X,\theta)-\mathcal H^{(J_{\rm prof})}_{h,\delta}(X,\theta)\right|
 \le \frac{|c_\alpha|\e^{X/\alpha-w_{J_{\rm prof}+1}/\kappa}}{\cos(|\theta|/2)}
 \left(\frac{h}{2w_{J_{\rm prof}+1}}+\kappa\right).
\end{equation}
The selected candidates are also evaluated at 90 digits with cutoffs 100 and 120. This checks series truncation and arithmetic at fixed points only. In particular, \eqref{eq:profile-tail} does not control the part of the $X$ axis outside either search window. No decay of the balanced profile as $X\to+\infty$ is assumed.
\begin{table}[!htbp]
 \centering\small
 \caption{Profile search results. $\widehat C$ is a numerical finite-window maximum and $\widehat X$ its selected coordinate. The last two columns are the relative change between windows and the bound \eqref{eq:profile-tail} at the selected point for cutoff 100. Neither is a global-supremum certificate.}\label{tab:profile-summary}
 \begin{tabular}{@{}cccccl@{}}\toprule
 Family & $\sigma/\sigma_*$ & $\widehat C$ & $\widehat X$ & Window change & Series-tail bound\\\midrule
 \ProfileSummaryRows\bottomrule
 \end{tabular}
\end{table}
\FloatBarrier

\subsection{Finite-sequence dimensions and scaled locations}
Table~\ref{tab:profile-finite} lists the actual counts and the coordinate $X_N=T-\alpha t_N$ of the selected boundary maximum, where $z_N$ denotes that boundary maximizer and $t_N=-\log|z_N|$. The order-one values of $X_N$ explain why the stationary profile can already be visible for modest $N$: the exact identity (3.21) has removed $T$, and the searched peak is already inside $X\le T$. This observation does not establish that the global profile supremum is attained at a finite $X$.
\begin{table}[!htbp]
 \centering\small
 \caption{Overresolved finite hybrid sequences. Counts include all prescribed basis functions: $L=K+n+1$. The normalized error is $\e^T E_N$, not $\e^{\alpha\sigma\sqrt L}E_N$.}\label{tab:profile-finite}
 {\setlength{\tabcolsep}{4pt}\begin{tabular}{@{}cccccccr@{}}\toprule
 Family & $\sigma/\sigma_*$ & $N$ & $K$ & $n$ & $L$ & $\e^T E_N$ & $X_N$\\\midrule
 \ProfileFiniteRows\bottomrule
 \end{tabular}}
\end{table}
\FloatBarrier
For family B at $N=36$, both spacings are also checked on the fine boundary grid; the individual changes are retained in the underlying computational records.  The profile series and the finite LP are evaluated by different formulas in the verification implementation: the first evaluates the $T$-free series directly, whereas the second constructs the finite hybrid LP before evaluating its boundary error.  They share the same mathematical parameters and are therefore a consistency cross-check rather than two unrelated software stacks.

The finite-realization discrepancy can be checked at a fixed profile point.
For each family and spacing, use the selected profile coordinate
$(\widehat X,\widehat\theta)$, with $\widehat\theta=\Theta$ in these four cases, and set
\[
 \widehat z_N=\exp\!\left((\widehat X-T)/\alpha+\ii\widehat\theta\right),
 \qquad
 E_N^{\rm pt}=\left|\widehat z_N^\alpha-
 \widetilde r^{\rm hyb}_{N,\delta}(\widehat z_N)\right|.
\]
Here $T=\alpha\sigma\sqrt N$ and $\widetilde r^{\rm hyb}_{N,\delta}$ is the
finite overresolved hybrid from Theorem~1.2. Write
$\widehat C=|\mathcal H_{h,\delta}(\widehat X,\widehat\theta)|$ for the
profile value at that same point. Table~\ref{tab:profile-discrepancy} reports
$|\e^T E_N^{\rm pt}-\widehat C|/\widehat C$ from the saved 80-digit values.
The coordinate $(\widehat X,\widehat\theta)$ is held fixed as $N$ varies.
Thus this table checks finite realization at a common profile point; it does
not compare independently searched maxima or quantify the error in the global
profile supremum. The searched errors $E_N$ in Table~\ref{tab:profile-finite}
retain their separate meaning.
\begin{table}[!htbp]
 \centering\small
 \caption{Fixed-point relative discrepancy $|\e^T E_N^{\rm pt}-\widehat C|/\widehat C$ from the saved 80-digit records. Both values use the same selected profile coordinate; these are pointwise finite-realization checks, not discrepancies between searched maxima.}\label{tab:profile-discrepancy}
 \begin{tabular}{@{}ccccc@{}}\toprule
 Family & $\sigma/\sigma_*$ & $N=36$ & $N=64$ & $N=100$\\\midrule
 A & $3/4$ & $4.92\times10^{-17}$ & $5.42\times10^{-22}$ & $6.82\times10^{-27}$\\
 A & $1$   & $7.10\times10^{-22}$ & $1.86\times10^{-28}$ & $6.21\times10^{-35}$\\
 B & $3/4$ & $6.48\times10^{-15}$ & $3.18\times10^{-19}$ & $1.63\times10^{-23}$\\
 B & $1$   & $5.27\times10^{-19}$ & $1.06\times10^{-24}$ & $2.40\times10^{-30}$\\\bottomrule
 \end{tabular}
\end{table}
\FloatBarrier
These pointwise checks support the finite-realization step; they neither
replace Theorem~1.2 nor certify its global constant.

\section{Numerical provenance and scope}\label{sec:archive}
The numerical macros and figures accompanying the article preserve the reported
experiments. Their construction and search specifications are given above. The
verification uses 80-digit decimal arithmetic for selected profile and
finite-construction values. The matched-space checks use binary64 solves and
extended-precision evaluation; they also compare SVD cutoffs and the two
pole-evaluation conventions. These checks support the displayed results within
their documented scope, without certifying continuous suprema or replacing the
proofs.\smallskip

\textit{Declarations.} The AI-use declaration in the main article applies equally to this supplementary material. The authors assume responsibility for all content.